\documentclass[12pt]{amsart}
\usepackage{geometry} 
\usepackage{amsmath,amsfonts,amssymb,mathabx}
\usepackage{accents}
\usepackage{stmaryrd} 
 \usepackage{amsaddr}
 \usepackage{mathrsfs,txfonts} 

\usepackage{dsfont}
\usepackage[utf8x]{inputenc}
\usepackage{enumitem}
\usepackage[T1]{fontenc}
\usepackage{lmodern} 
\usepackage{fullpage}
\usepackage{graphicx}
\usepackage{caption}
\usepackage{subcaption}

\usepackage{array}
\usepackage{multicol}
\usepackage{multirow}
\usepackage{color}
\usepackage{esint}
\usepackage{cancel}
\usepackage{pgfplots}
\pgfplotsset{compat=newest}
\usepgfplotslibrary{fillbetween} 

\usepackage{tikz}
\usetikzlibrary{decorations}
\usetikzlibrary{decorations.pathreplacing,calc}
\usetikzlibrary{patterns}
\usetikzlibrary{arrows.meta} 
\usetikzlibrary{plotmarks}

\numberwithin{equation}{section}

\newcommand{\R}{\mathbb{R}}
\newcommand{\dd}{\mathrm{d}}
\newcommand{\DD}{\mathrm{d}}

\newcommand{\W}{\mathscr{W}} 
\newcommand{\U}{\mathscr{U}} 

\newcommand{\T}{\mathscr{T}}
  
\newcommand{\n}{\mathrm{n}}
\renewcommand{\r}{\mathrm{r}}
\newcommand{\p}{\mathrm{p}}
\newcommand{\g}{\mathrm{g}_\kappa}

\newcommand{\h}{\mathrm{h}}
\newcommand{\f}{\mathrm{f}}

\newcommand{\y}{\mathrm{y}}
\newcommand{\z}{\mathrm{z}}

\newcommand{\ud}{\, {\mathrm{d}}}

\newcommand{\eps}{\epsilon}

\newtheorem{theorem}{Theorem}[section]
\newtheorem{lemma}[theorem]{Lemma}
\newtheorem{definition}[theorem]{Definition}
\newtheorem{rmk}[theorem]{Remark}
\newtheorem{prop}[theorem]{Proposition}

\newtheorem{exo}[theorem]{Example}
 
\title{
Shock profiles for hydrodynamic models for fluid-particles flows in the flowing regime
}
\author[1]{Thierry Goudon}\address[1]{Universit\'e C\^ote d'Azur, Inria,  CNRS,\\
	LJAD, Parc Valrose, F-06108 Nice, France}
\author[2]{Pauline Lafitte}\address[2]{CentraleSup\'elec, F\'ed\'eration de Math\'ematiques FR CNRS 3487\\ 
	\& Labo.~MICS, F-91192 Gif-sur-Yvette, France}
\author[3]{Corrado Mascia}
\address[3]{Dipartimento di Matematica Guido Castelnuovo,\\
	Sapienza, University of Rome, Italy}
\begin{document}
\baselineskip16pt

\begin{abstract}
Starting from coupled fluid-kinetic equations for  the  modeling of laden flows, we derive relevant viscous corrections to be added to
asymptotic hydrodynamic systems, by means of Chapman-Enskog expansions and analyse the shock profile structure for such 
limiting systems.
Our main findings can be summarized as follows.
Firstly, we consider  simplified  models, which are intended to reproduce
the main difficulties and features of more intricate systems.
However, while they are more easily accessible to analysis, such toy-models should be considered with caution
since they might lose many important structural properties of the more realistic systems.
Secondly, shock profiles can be identified also in such a case, which can be proven to be stable at least in the
regime of small amplitude shocks.
Last, but not least, regarding at the temperature of the mixture flow as a parameter of the problem,
we show that the zero-temperature model admits viscous shock profiles.
Numerical results indicate that a similar conclusion should apply in the regime of small positive temperatures.
\end{abstract}

\maketitle 


\thispagestyle{empty}

\vspace*{.2cm}
{\small
\noindent{\bf Math.~Subject Classification.} 
35C07 
35L65 
35Q35 
76L05 
}

\tableofcontents

\section{Introduction}

A {\it particle-laden flow} is a class of two-phase fluid flow composed of a {\it carrier phase},
the surrounding
   continuous medium, and a {\it disperse phase}, constituted of small, immiscible
and dilute particles.
Such flows occur in many natural phenomena and industrial processes: snow and rock avalanches \cite{Man,Man2},
desert sandstorms, dispersions of pollutants, pollen and allergens in air \cite{Mora}, aerosols in respiratory flows \cite{BBJM, Moussa},
fluidised beds \cite{HC}, fuel injector, chemical reactors, internal combustion engines \cite{Kiva,Hyk,williams}, just to name a few.

The broad variety of applications, and the wide range of scales involved in these situations,  make it difficult to develop a unified framework.
Two main viewpoints have been adopted to model such flows.
The so-called {\it Eulerian approach} considers all phases  as a continuum so that one is led to hydrodynamic systems
for the densities and velocities (at least) of the disperse phase and the carrier phase \cite{Ba, Gida, Ish}.
In contrast, the {\it Lagrangian approach} describes the particles by means of
their distribution function in phase space, 
the evolution of which is coupled  to a hydrodynamic model, based on either Euler or Navier-Stokes equations, for the carrier fluid.
This defines a fluid-kinetic framework for describing the laden flow under consideration \cite{Orou,PJ1}.
In both cases, the coupling is mainly achieved  through the drag forces exerted by a phase on the other,
which induces momentum exchanges between the two phases.
A valuable approach consists in bringing out connections between these different settings,
following the derivation of fluid equations from the  kinetic equations of gas dynamics \cite{LSR}:
several asymptotic regimes have been identified and investigated, both on theoretical and numerical grounds
 \cite{DV,GJV1,GJV2, Ham,Hof,PEJ,pej1}. The present work is a contribution in this direction.
 
As stated above, an alternative to the continuum approach describes the disperse phase by means
of a Fokker-Planck equation for the dimensionless particle distribution function $f_\eps:(t,x,v)\mapsto f_\eps(t,x,v)$, that is
\begin{equation*}
	\dfrac{1}{T_{\textrm{\tiny ref}}}	\partial_ t f_\eps+\dfrac{V_{\textrm{\tiny ref}}\,v}{L_{\textrm{\tiny ref}}}\partial_x f_\eps
 	= \dfrac{1}{T_S V_{\textrm{\tiny ref}}}\partial_v \left\{V_{\textrm{\tiny ref}}(v-u_\eps)f_\eps+\dfrac{V_{\textrm{\tiny th}}^2}{V_{\textrm{\tiny ref}}}\partial_v f_\eps\right\},\,
\end{equation*}
where
\begin{itemize}
\item $t>0$, $x\in\R$ and $v\in\R$ are the dimensionless time, position and velocity variables, respectively;
\item $T_{\textrm{\tiny ref}}$, $L_{\textrm{\tiny ref}}$ and $V_{\textrm{\tiny ref}}:=L_{\textrm{\tiny ref}}/T_{\textrm{\tiny ref}}$
	are the time, position and velocity dimensions, respectively;
\item $u_\eps$ is the velocity of the surrounding medium  (dimensionless with respect to $V_{\textrm{\tiny ref}}$);
\item the Stokes settling time $T_S$ and the thermal speed $V_{\textrm{\tiny th}}$ are defined by
	\begin{equation*}
        		T_S:=\frac{m}{6\pi \mu a}\qquad\textrm{and}\qquad V_{\textrm{\tiny th}}:=\sqrt{\frac{\kappa_B \Theta}{m}}\,,
        \end{equation*}
where $a$ and $m$ are the radius and mass of the particles,
$\mu$ and $\Theta$ are the dynamic viscosity and temperature of the surrounding fluid,
and $\kappa_B$ is the Boltzmann constant.
\end{itemize}
In the following, we concentrate on the flowing regime where $T_S=\eps\,T_{\textrm{\tiny ref}}$ and ${V_{\textrm{\tiny ref}}= V_{\textrm{\tiny th}}\sqrt{\theta}}$.
The parameters $\eps$ and $\theta$ are the reminders of the process of making the equation dimensionless.
Moreover, we focus on the regime $\eps$ small, viz. $0<\eps\ll 1$.
We refer the reader to \cite{CG, CGL} for further details on this scaling.

The resulting equation for the unknown $f_\eps$, describing the particle distribution in the phase space, is
\begin{equation}\label{eq:kin}
  	\partial_ t f_\eps+v\partial_x f_\eps =\dfrac{1}{\eps} L_{u_\eps}(f_\eps)\,,
\end{equation}    
with the {\it Fokker--Planck operator} $L_u$ defined by
\begin{equation}\label{eq:fp}
	L_{u} (f):= \partial_v \bigl\{(v-u)f+\theta \partial_v f\bigr\}
\end{equation}
Since $u_\eps$ represents the velocity of the surrounding medium, the term $\partial_v \left\{(v-u_\eps)f_\eps\right\}$
describes the drag force exerted on the particles by the fluid, assumed to be proportional to the relative velocity between the two species.
Taking the zero-th and first order moments  over the velocity variable gives the apparent mass density of particles $\rho_\eps$
and momentum of the disperse phase $J_\eps$, where 
\begin{equation*}
	\rho_\eps(t,x):=\int f_\eps(t,x,v)\ud v,\qquad J_\eps(t,x):=\int v  f_\eps(t,x,v)\ud v.
\end{equation*}
Equation \eqref{eq:kin} is coupled to a balance law for the momentum of the carrier phase
\begin{equation}\label{eq:fl1}
	\partial_ t (n_\eps u_\eps) +\partial_x \left\{n_\eps u_\eps^2+p(n_\eps)\right\}
		=\dfrac{1}{\eps} (J_\eps-\rho_\eps u_\eps),
\end{equation}
where $n_\eps$ and $u_\eps$ are, respectively, the mass density and the velocity
field of the fluid.
We assume $n_\eps$ is already  dimensionless with respect to a reference density
$n_{\textrm{\tiny ref}}$ and also make $\rho_\eps$ dimensionless with respect to $n_{\textrm{\tiny ref}}$. In
the same way, $p$, that describes the pressure of the carrier phase, is
already supposed dimensionless being defined by
\begin{equation*}
	p(n):=\frac{\tilde{p}(n_{\textrm{\tiny ref}}n)}{n_{\textrm{\tiny ref}}\,V_{\textrm{\tiny ref}}^2}\,,
\end{equation*}
where $\tilde{p}$ is the dimensionalized pressure.
The right-hand side in \eqref{eq:fl1} accounts for the back-friction force exerted by the particles on the fluid.

 Some hypotheses are required on the function $n\mapsto p(n)$.
Precisely, we assume that $p\in C^2$ is a strictly increasing, convex and coercive function, i.e.
\begin{equation}\label{eq:hyppressure}
	p',\,p''>0\quad\textrm{in $(0,\infty)$}\quad\textrm{and}\quad
	\lim_{n\to+\infty}  \displaystyle\frac{p(n)}{n}=+\infty.
\end{equation}
Since the pressure is determined up to an additive constant, we assume the additional condition $p(0)=0$.
Moreover, we focus on the case $p'(0)=0$, a relevant case being the standard pure power form,
usually referred to as  {\it $\gamma$-law},
\begin{equation}\label{eq:gammalaw}
	p(n):=Cn^\gamma\quad\textrm{with}\quad C>0,\quad \gamma>1.
\end{equation}
As $\eps\to 0$ in \eqref{eq:kin}, we guess that 
\begin{equation}\label{eq:ansatz}
	f_\eps(t,x,v)\simeq \rho_\eps (t,x)M_{u_\eps(t,x)}(v)\,,
 \end{equation}
where $M_u$ is the standard {\it Maxwellian distribution},   defined by
\begin{equation}\label{eq:maxwell}
	M_u(v):=\frac{1}{\sqrt{2\pi\theta}}\exp\left(-\frac{|v-u|^2}{2\theta}\right).
\end{equation}      
Since $\theta\partial_v M_u=-(v-u)M_u$, the Fokker--Planck operator $L_u$ can be rewritten as
\begin{equation}\label{eq:fp_form}
	L_u(f)=\theta \partial_v\left\{M_u\,\partial_v(M_{u}^{-1}f)\right\},
\end{equation}
showing, in particular, that $L_u$ vanishes when computed at  $v\mapsto f(v)=\rho M_{u}(v)$.
As a consequence, we expect that the dynamics can be described by means of macroscopic quantities in such a regime.
Indeed, integrating \eqref{eq:kin} with respect to velocity yields
\begin{equation*}
	\partial_t \rho_\eps+\partial_x J_\eps=0.
\end{equation*}
Next, we add the equation for the first order moment to \eqref{eq:fl1} in order to get rid of the singular term by using the identity
\begin{equation*}
	\int v\,\partial_v L_{u_\eps}(f_\eps)\ud v= - \int \bigl\{(v-u_\eps)f_\eps+\theta \partial_v f_\eps\bigr\}\ud v= - J_\eps + \rho_\eps u_\eps.
\end{equation*}
Hence, we end up with
\begin{equation*}
	\partial_ t ( J _\eps+ n_\eps u_\eps) +\partial_x \biggl\{\int v^2 f\ud v + n_\eps u_\eps^2+p(n_\eps)\biggr\} =0.
\end{equation*}
Going back to the ansatz \eqref{eq:ansatz}, we infer
\begin{equation}\label{eq:flux_secondmoment}
	J_\eps\simeq \rho_\eps u_\eps, \qquad \int v^2 f_\eps\ud v\simeq \rho_\eps u _\eps^2+\theta \rho_\eps,
\end{equation}
and, dropping the dependence with  respect to $\eps$, we get the first order system 
\begin{equation}\label{eq:sys}
	\left\{\begin{array}{l}
	\partial_t \rho+\partial_x(\rho u ) =0,\\
	\partial_ t (ru) +\partial_x \left\{ru^2+p(n)+\theta \rho \right\}=0.
	\end{array}\right.
\end{equation}
where $r:=\rho+n$ is called {\it hybrid density}, being the sum of the densities of the disperse
and the carrier phases, denoted by $\rho$ and $n$, respectively.

From the modeling viewpoint, in some circumstances,
it might be questionable to consider the diffusion with respect to the velocity variable as a stiff term in equation \eqref{eq:kin}.
Thus, it  is  equally relevant to consider the situation where $\theta=0$, which means that the Brownian velocity 
fluctuations are negligible.
This situation is much more difficult for the analysis, since the formal ansatz becomes singular.
Namely, as $\eps\to 0$, denoting by  $\delta_{v=u}$ the Dirac delta centered at $u$, we formally infer
 \begin{equation*}
	f_\eps(t,x,v)\simeq \rho_\eps(t,x)\,\delta_{v=u_\eps(t,x)}
\end{equation*}
which leads to \eqref{eq:flux_secondmoment} with $\theta=0$.
This approximation is often used in the modeling of laden flows, but depending on the considered coupling or asymptotic regime,
this pressureless regime might lead to difficulties, both for the analysis \cite{Hof, PEJ,pej1} and for numerics, and possibly
to physically  irrelevant results \cite{saurel}.
Nevertheless, in this paper, we also consider the system \eqref{eq:sys} with $\theta=0$, regarded as a (formal) limiting regime.

Still inspired by the kinetic theory of gases, our objectives are the following.
First, we formally derive \emph{diffusive corrections} to system \eqref{eq:sys} coupled with \eqref{eq:fl1},
in the same spirit as the Chapman-Enskog procedure  leads to the Navier-Stokes equations, keeping track of the $\mathscr O(\eps)$-viscosity terms.
Second, we investigate the structure of viscous shock profiles for the obtained systems.
Namely, following the pioneering work \cite{gilbarg},  we wish to identify solutions of the diffusive equations
with the form
\begin{equation*}
	(\rho,n,u)(t,x)=\mathrm W(y)
	\quad\textrm{where}\quad y:=x-ct,
\end{equation*}
for some given profile $\mathrm W$ with prescribed far-end states, that correspond to ``admissible'' discontinuous 
solutions of the diffusion-less system.

As a warm-up, we start with the case where \eqref{eq:fl1} reduces to the mere Burgers equation: namely in \eqref{eq:fl1},
we  (brutally) set $n_\eps=1$. Hence, we firstly  approach system \eqref{eq:kin}-\eqref{eq:fl1}
with the {\it inviscid Burgers fluid-particle system}, given by 
\begin{equation}\label{eq:red}\tag{\texttt{iB}}
	\partial_t\begin{pmatrix}\rho \\ ru \end{pmatrix}+ \partial_x\begin{pmatrix} \rho u \\ ru^2+\theta \rho \end{pmatrix} = 0,
\end{equation}
recalling that $r=1+\rho$, and its corresponding viscous correction, referred to as the {\it viscous Burgers fluid-particle system}, whose explicit form is
\begin{equation}\label{eq:bfp}\tag{\texttt{vB}}
	\partial_t\begin{pmatrix}\rho \\ ru \end{pmatrix}
	+ \partial_x\begin{pmatrix} \rho u \\ ru^2+\theta \rho \end{pmatrix}
	=\eps\partial_x\left(\mathbf{D}(\rho,ru)\partial_x \begin{pmatrix} \rho \\ ru \end{pmatrix}\right)\,,
\end{equation}
where
\begin{equation}\label{eq:diffBfp}
	\mathbf{D}(\rho,ru)
	=\frac{\rho u}{r^3}\begin{pmatrix} u & -1 \\ 0  & 0 \end{pmatrix}
	+\frac{\theta}{r}\begin{pmatrix} 1/r & 0 \\ -\rho u  & \rho \end{pmatrix}
\end{equation}
(the formal derivation of the correction terms of order $\eps$ will be detailed later on).
Even if both \eqref{eq:red} and \eqref{eq:bfp} possess an entropy $\zeta$, defined by
\begin{equation*}
 \zeta(\rho,ru):=\tfrac12ru^2 +\theta\rho\ln\rho,
\end{equation*}
such toy models are not fully physically meaningful, the main criticism being that they are not invariant under Galilean transformations.
Nevertheless, they are considered here because they are amenable to detailed computations,
which we consider illuminating.

Next, we move to the  coupling with the Euler equations, where the density of the carrier fluid is driven
by the additional conservation law 
\begin{equation*}
	\partial_t n_\eps+\partial_x(n_\eps u_\eps)=0.
\end{equation*}
The corresponding {\it inviscid Euler fluid-particle system} is
\begin{equation}\label{eq:ext}\tag{\texttt{iE}}
	\partial_t\begin{pmatrix} r \\ \rho \\ r u \end{pmatrix}
	+ \partial_x\begin{pmatrix} r u \\ \rho u \\ ru^2+p(n)+\theta \rho \end{pmatrix}=0\,,
\end{equation}
and the higher-order correction, named {\it viscous Euler fluid-particle system} is
\begin{equation}\label{eq:efp}\tag{\texttt{vE}}
	\partial_t\begin{pmatrix} r \\ \rho \\ r u \end{pmatrix}
	+ \partial_x\begin{pmatrix} r u \\ \rho u \\ ru^2+p(n)+\theta \rho \end{pmatrix}
	=\eps\partial_x\left(\mathbf{D}(r,\rho,ru)\,\partial_x \begin{pmatrix} r \\ \rho \\ ru \end{pmatrix} \right),
\end{equation}
where 
\begin{equation}\label{eq:diff_Efp}
	\mathbf{D}(r,\rho,ru)
	=\frac{\rho\,np'(n)}{r^2}\begin{pmatrix} 0 & 0 & 0 \\ -1  & 1 & 0 \\ 0 & 0 & 0 \end{pmatrix}
	+\theta \begin{pmatrix} 0 & 0 & 0 \\ 0  & n^2/r^2 & 0 \\ -{\rho u}/{r} & 0 & {\rho}/{r} \end{pmatrix}
\end{equation}
(again, the formal derivation will be detailed later on).
Differently from the previous case, systems \eqref{eq:ext} and  \eqref{eq:efp} are {\it invariant under Galilean transformations}.
In addition, \eqref{eq:efp} also possesses an entropy, defined by
\begin{equation*} 
	 \zeta(r,\rho,ru):=\tfrac12ru^2 +\Pi(n)+\theta\rho\ln \rho
\end{equation*}
where
\begin{equation*}
	\Pi(n):=\int_0^n \int_0^s \frac{1}{\varsigma}\frac{\ud p}{\ud\varsigma}(\varsigma)\,\ud \varsigma \ud s.
\end{equation*}
In general, for both \eqref{eq:bfp} and \eqref{eq:efp}, the  existence of an entropy $\zeta$ plays a pivotal role;
specifically, it will be crucial to establish existence (and stability) of viscous shock profiles.

The paper is organized as follows.
Section~\ref{sec:Gal} collects some useful notions and basic facts on general hyperbolic-parabolic systems.
It can be safely skipped by the reader familiar with these topics.
In Section~\ref{sec:Bur}, we consider the model \eqref{eq:bfp}, establishing the existence of viscous profile
for weak shocks with positive temperatures.
Subsequently, in Section~\ref{sec:Euler}, we turn to analyze  system \eqref{eq:efp} where the diffusion correction term is degenerate.
Nevertheless, we are still able to provide a rigorous proof for the existence of weak shock profiles, 
whose stability can be established by appealing to  general results for small-amplitude profiles of hyperbolic-parabolic systems.
We also  investigate the case where $\theta=0$, which induces new degeneracies;
in particular, the entropy of the system is not strictly convex.
Finally, Section~\ref{num} is devoted to further studying the model \eqref{eq:efp} starting from the basic observation  
that a more complete result can be obtained for the temperature-less system, proceeding by direct inspection of the
corresponding ODE.
Expressing the ODE in reduced variables allows  us to show that there are in fact two parameters of interest.
This leads to showing the existence of a shock profile, which is illustrated numerically.
In the temperature case, the differential system is also expressed in these reduced variables and solved
numerically for small temperatures.
Finally, the numerical profile is compared to its temperature-less counterpart.
 
\section{General properties of conservation laws}\label{sec:Gal}

Let us collect here a series of definitions and basic statements that will be used throughout the paper.
For further details, we refer the reader to the classical textbooks \cite{Daf,Smo}.
Let $\mathscr M_m(\mathbb R)$ be the space of $m\times m$ matrices with real entries.
Then, given functions $F:\mathbb R^m\rightarrow \mathbb R^m$ and $\mathbf D: \mathbb R^m\rightarrow \mathscr M_m$,
we consider the system of conservation laws for the unknown function $\W:[0,\infty)\times\mathbb R\rightarrow \mathbb R^m$
\begin{equation}\label{eq:hyppar_cons}
	\partial_t \W+\partial_x F(\W)
	=\eps \partial_{x}\left\{\mathbf{D}(\W) \partial_x \W\right\}
	\qquad t\geq 0,\quad x\in \mathbb R,
\end{equation}
for some $\eps>0$ under the assumption that the formal limiting system $\eps\to 0^+$
\begin{equation}\label{eq:hyp_cons}
	\partial_t \W+\partial_x F(\W)=0
	\qquad t\geq 0,\quad x\in \mathbb R,
\end{equation}
is strictly hyperbolic, i.e. the Jacobian $\dd F$ has real distinct eigenvalues for any $\W$ under consideration.

\begin{definition}
Let $\mathbf{A},\mathbf{B}\in\mathscr M_n$ two matrices with $\mathbf{B}$ invertible.
A (column) vector  $\mathbf{r}\neq 0$ is said to be a \emph{right
  eigenvector} of $\mathbf{A}$ with respect to $\mathbf{B}$ 
relative to the eigenvalue $\lambda$ if there holds $\left(\mathbf{A}-\lambda\mathbf{B}\right)\mathbf{r}=0$.
A \emph{left (row) eigenvector} $\boldsymbol{\ell}\neq 0$ of $\mathbf{A}$
with respect to $\mathbf{B}$ relative to the eigenvalue $\lambda$ is
 defined as $\boldsymbol{\ell}\left(\mathbf{A}-\lambda\mathbf{B}\right)=0$.
\end{definition}

For shortness, we use the shortened names {\it right/left eigenvector of $\mathbf{A}$ with respect to $\mathbf{B}$} whenever
the eigenvalue $\lambda$ is clear from the context.

To start with, we state and prove a straightforward Lemma showing that the directional derivatives of the eigenvalues
of $\dd F$ with respect to the corresponding right eigenvectors are invariant under diffeomorphisms.

\begin{lemma}\label{lem:lemm_gal}
Let $F$, $G$, $H\,:\,\mathbb R^m\rightarrow \mathbb R^m$ be three differentiable functions
such that $\dd G$ is invertible and $F=H\circ G^{-1}$.
Let  $\lambda$  be an eigenvalue of $\mathrm dF(\W)$,  or, equivalently, an eigenvalue of $\mathrm dH(\U)$
with respect to $\mathrm dG(\U)$, where $\W=G(\U)$.
Let $\mathbf{r}$ be a right eigenvector of $\dd F$ with respect to
$\mathbf{I}$. Then $\mathbf{s}=\dd G(\U)^{-1}\mathbf{r}$ is a right
eigenvector of $\dd H$ with respect to $\dd G$, also for the eigenvalue $\lambda$.
Moreover, the scalar products $\nabla_{\W} \lambda\cdot \mathbf{r}$ and
$\nabla_{\U} \mu\cdot \mathbf{s}$, where $\mu(\U)=\lambda(G(\U))$, coincide.
\end{lemma} 
\begin{proof}
Let $H(\U):=F\circ G(\U)=F(\W)$.
The statement is a consequence of the chain rule which leads to the identities
\begin{equation*}
	\dd F(\W)=\dd H(G^{-1}(\W))\ \dd (G^{-1})(\W),\qquad
	\dd (G^{-1})(\W)=\big(\dd G(\U)\big)^{-1}\,,
\end{equation*}
with the former recast simply as $\dd F(\W)=\dd H(\U)\ \dd G(\U)^{-1}$.
For $(\lambda, \mathbf{r})$ a left eigenpair of the matrix $\dd F$, we obtain
\begin{equation*}
	0=\bigl(\dd F(\W)-\lambda\mathbf{I}\bigr)\mathbf{r}
	=\bigl(\dd H(\U) \dd G(\U)^{-1}-\lambda\mathbf{I}\bigr)\mathbf{r}
	=\bigl(\dd H(\U)-\lambda \dd G(\U)\bigr)\mathbf{s}
\end{equation*}
with $\mathbf{s}:=\dd G(\U)^{-1}\mathbf{r}.$
Similarly, if $\boldsymbol{\ell}$ is a left eigenvector of $\dd F(\W)$, we get
\begin{equation*}
	0=\boldsymbol{\ell}\bigl(\dd F(\W)-\lambda\mathbf{I}\bigr)
	=\boldsymbol{\ell}\bigl(\dd H(\U)-\lambda\dd G(\U)\bigr)\dd G(\U)^{-1}.
\end{equation*}
Thus, we infer that  $\boldsymbol{\ell}$ is a left eigenvector of $\dd H$ with respect to $\dd G$. 
Next, we compute the gradient of the eigenvalue $\lambda(\W)=\lambda(G(\U))=\mu(\U)$ with respect to the 
variables $\W$  (conservative) and  $\U$ (non conservative) obtaining
\begin{equation*}
	\nabla_{\U} \mu (\U)=\dd G(\U)^\intercal \, \nabla_{\W} \lambda (G(\U))\,.
\end{equation*}
Hence, there holds
\begin{equation*}
	\nabla_{\W} \lambda(\W)\cdot \mathbf{r}
		=\nabla_{\U} \mu(\U) \cdot \dd G(\U)^{-1}\mathbf{r}
		=\nabla_{\U} \mu(\U) \cdot \mathbf{s},
\end{equation*}
which concludes the proof.
\end{proof}

The condition $\nabla_{\W} \lambda \cdot \mathbf{r}\neq 0$ characterizes \emph{genuinely nonlinear} fields.
It plays the same role as strict convexity for scalar conservation laws, see \cite[Section~17-B]{Smo}.
Oppositely, {\it linearly degenerate} fields, defined as the ones for which $\nabla_{\W} \lambda\cdot \mathbf{r}= 0$ holds,
correspond to linear transport equations with a  pure motion of the initial datum without gain and loss of regularity.
In particular, asymptotically stable shock solutions cannot be expected to appear into play.

\subsection{Shock wave solutions}

In the limiting regime $\eps=0$, we are specifically interested in discontinuous solutions that reach a specific state $\W_\ast$,
which are required to satisfy the classical {\it Rankine-Hugoniot conditions} \cite{Hug2,Hug3, Rank}
\begin{equation}\label{eq:gen_rh}
	c\, [\![ \W ]\!]=[\![ F(\W) ]\!]\,,
\end{equation}
where $c$ is the jump speed and $ [\![ \W ]\!]:=\W-\W_\ast$.
Such solutions can be  parameterized by the scalar quantity $s\geq 0$  and they are described by curves $s\mapsto \W(s)$
with speed function $s\mapsto c(s)$ associated to the eigenpairs of $\dd F$ such that
\begin{equation*}
	\left\{\begin{aligned}
		\W(0)&=\W_\ast\\
		\dot {\W}(0)&=\mathbf r(\W_\ast)
	\end{aligned}\right.
	\quad\textrm{and}\quad
	\left\{\begin{aligned}
		c(0)&=\lambda(\W_\ast)\\
		\dot c(0)&=\tfrac12 \lambda(\W_\ast)\mathbf r(\W_\ast)
	\end{aligned}\right.
\end{equation*}
(see e.g. \cite[Section~8.2]{Daf} or \cite[Section~17-B]{Smo}).

These pure jump solutions are said to satisfy {\it Liu's entropy criterion} when
\begin{equation}\label{eq:gen_liu}
	c(s)\leq c(\sigma) \textrm{ holds for any $0\leq \sigma\leq s$}.
\end{equation}
The above criterion is crucial because it can be used to select relevant
solutions among all weak discontinuous solutions of the equation.
We refer the reader to \cite{Daf} for motivations and technical details about the conditions, which date back to \cite{Liu}.

\subsection{Stability concepts}

Next, let us switch on the diffusive term in system \eqref{eq:hyppar_cons} by considering the case  $\epsilon>0$.
As a starting point, we consider the initial value problem for the linearized system at the state $\W_\ast$,
namely
\begin{equation}\label{eq:sys_lin}
	\partial_t \W_\eps +\mathbf{A} \partial_x \W_\eps=\eps \mathbf{D}\partial_x^2 \W_\eps\,,
	\qquad \W_\eps(0,\cdot)=\W_{\eps,0}(\cdot)\,,
\end{equation}
where $\mathbf{A}:=\dd F(\W_\ast)$ and $\mathbf{D}:={\mathbf D}(\W_\ast)$.

System \eqref{eq:sys_lin} has constant coefficients and, consequently, it can be scrutinized by means of standard Fourier analysis,
analysing the corresponding symbol $P^\eps_\ast(\xi):=i\xi\mathbf{A}+\eps\,\xi^2{\mathbf D}$.
As it is well-known, the Fourier transform $\hat{\W}_\eps$ of ${\W}_\eps$ solves
$\partial_t \hat{\W}_\eps = -P^\eps_\ast(\xi)\hat{\W}_\eps$ with initial condition $\hat{\W}_\eps(0)=\hat{\W}_{\eps,0}$,
whose solution $\hat{\W}_\eps=\hat{\W}_\eps(t;\xi)$ is formally given by the operator $t\mapsto \exp\{-t P^\eps_\ast(\xi)\}\hat{\W}_{\eps,0}$.

%
%

In \cite{MaPe,Pego84} different stability notions have been introduced, which turn out to be crucial for the existence of shock profiles.

\begin{definition}\label{def:def2}
The linear system \eqref{eq:sys_lin} is \emph{uniformly stable at $\W_\ast$ with respect to $\eps$},
or simply \emph{stable at $\W_\ast$}, if for any $T>0$
there exists $C_T>0$, independent of $\eps$, 
such that 
\begin{equation*}
	\sup\left\{\frac{\|\W_\eps(t,\cdot)\|_{L^2}}{\|\W_{\eps,0}\|_{L^2}}\,:\, 0<\eps<1,\,t\in[0,T]\right\}\leq C_T\,.
\end{equation*}
for some initial datum $\W_{\eps,0}$ with non-zero $L^2$-norm.
The set of stable linear systems \eqref{eq:sys_lin} is denoted by $\mathcal{S}$.
The interior of such set is composed by \emph{strictly stable systems}.
\end{definition}

Stability of \eqref{eq:sys_lin} can be rephrased by means of a property on the matrices ${\mathbf A}$ and  ${\mathbf D}$.  
Namely, according to \cite{Pego84}, one has to check that the matrix ${\mathbf D}$ is {\it uniformly stable} with respect to ${\mathbf A}$,
meaning that for each $T>0$ there exists a constant $C_T$ such that
\begin{equation}\label{eq:stablematrix}
\sup\left\{\left\|\exp\{-t P^\eps_\ast(\xi)\}\right\|_{\mathcal{M}_m}\,:\,0<\eps<1,\,t\in[0,T],\,\xi\in\mathbb{R}\right\}\leq C_T\,, 
\end{equation} 
where $\|\cdot\|_{\mathcal{L}(L^2)}$ denotes the operator norm from $L^2$ to $L^2$.
The latter is also equivalent to the existence of a universal constant $C>0$ such that
\begin{equation*}
\sup\limits_{t\geq 0,\,\zeta\in\R} \left\|\exp\{-t P^1_\ast(\zeta)\}\right\|_{\mathcal{M}_m}\leq C\,.
\end{equation*}
In \cite[Theorem 2.1]{MaPe} a list of properties equivalent to strict stability is given.
Among them, we recall the following one for readers' convenience.

\begin{theorem}\label{th:majdapego21}
The linear system  \eqref{eq:sys_lin} is strictly stable if and only if there exists $\delta>0$ such that
the eigenvalues  $\lambda_j(\xi)$ of the symbol $P^\eps_\ast(\xi)$ satisfy the condition
 \begin{equation*}
 	\textrm{\rm Re}\,\lambda_j(\xi)\leq -\delta|\xi|^2
	\qquad\textrm{for any}\quad \xi\in\R.
\end{equation*}
\end{theorem}

The above result induces a necessary and sufficient condition for strict stability which is
more manageable with respect to the original (and more abstract) definition.

\subsection{Entropy in the general setting}

A pivotal role is played by the notion of {\it entropy}, which provides very strong structural consequences on the underlying PDE system.

\begin{definition} 
Let $\mathcal{U}\subset\mathbb R^m$ be a neighborhood of some reference point $\W_\ast$.
The $C^2$ functions $\zeta: \mathcal{U} \to \mathbb R$ and $q:  \mathcal{U} \to \mathbb R$ with $\nabla q^\intercal=\nabla \zeta^\intercal\,\ud F$
form an \emph{entropy/entropy flux pair} for system \eqref{eq:hyppar_cons} if, for any $\W\in \mathcal{U}$,
\begin{itemize}
\item[\bf i.] {\sl (entropy convexity)} $\ud^2\zeta $ is positive definite;
\item[\bf ii.] {\sl (dissipativity)} $\ud^2\zeta\,\mathbf{D}$ has a positive definite symmetric part.
\end{itemize}
\end{definition}

Incidentally, let us note that a necessary condition for the existence of a function $q$ such that $\nabla q^\intercal=\nabla \zeta^\intercal\,\ud F$
is that the derivative of $\nabla \zeta^\intercal\,\ud F$ is symmetric.
In coordinates, this amounts to require 
\begin{equation*}
	(\ud^2 q)_{ij} = \partial_{j}\Bigl(\sum_{k} \partial_k \zeta_k\,\partial_i F_k \Bigr)
		=\sum_{k} \partial_k \zeta_k\,\partial_{ji}^2 F_k + \sum_{k} \partial_{jk}^2 \zeta_k\,\partial_{i} F_k.
\end{equation*}
Hence, $\ud^2 F_k$ being  symmetric, this is equivalent to requesting that $\ud^2\zeta \ud F$ is symmetric.

\begin{prop}\label{prop:ent_form}  
Assume system  \eqref{eq:hyppar_cons} admits a strictly convex entropy $\zeta$.
Then, the \emph{entropy variable} $\U := \nabla\zeta(\W)$ satisfies 
 \begin{equation}\label{eq:entr_sys}
	\partial_t G(\U) +\partial_x H(\U)=\eps\partial_x \left\{\mathbf B(\U)\partial_x\U\right\}
\end{equation}
where  $\W=G(\U)$, $\ud G$ is symmetric positive definite, $\ud H$  is symmetric, $\mathbf B$ is symmetric.
\end{prop}

\begin{proof} The change of coordinates $\W\to \U = \nabla \zeta (\W)$  is globally invertible, since
its Jacobian $\ud^2 \zeta$ is symmetric and positive definite.
In turn,  system \eqref{eq:hyppar_cons} can be cast under the form \eqref{eq:entr_sys}, with 
$\mathrm dG(\U)=(d^2\zeta(\W))^{-1}$  symmetric positive definite, since the entropy is strictly convex,
where $H(\U)=\left(F\circ G\right)(\U)$ and $\mathbf B(\U)=\left(\mathbf D\circ G\right)(\U)\,\mathrm d G(\U)=
 \mathbf D(\W)(\mathrm d^2 \zeta(\W))^{-1}$.
The symmetry of $\ud H = \ud F (\textrm{d}^2\zeta)^{-1}$, and ${ \mathbf B}$
follow from the symmetry of $\ud^2\zeta \ud H$ and $\ud^2\zeta\,\mathbf D$. 
\end{proof}

In addition, following \cite[Corollary 2.2]{MaPe}, it can be proved that {\it a sufficient condition for strict stability at $\W_\ast$
is the existence of a positive definite symmetric matrix $\mathbf{X}$ so that $\mathbf{X}\mathbf{A}$ is symmetric and $\mathbf{X}\mathbf{D}$ is
positive definite (not necessarily symmetric)}.
Later on, the matrix $\mathbf{X}$ will be chosen equal to the hessian $\ud^2\zeta$ of the entropy $\zeta$, i.e. $\mathbf{X}=\ud^2\zeta$.

\subsection{Energy estimates and viscous dissipation}
\label{sec:sub_enest}

The existence of an entropy is crucial to develop some basic energy estimates holding for \eqref{eq:hyppar_cons}.
For the sake of simplicity, let us  explain the  role of 
entropy by considering  the linearized equations in \eqref{eq:sys_lin}.

Preliminarily, let us recall a standard property.
Decomposing a (constant) matrix $\mathbf{A}$ as the sum of its symmetric and skew-symmetric parts
$\mathbf{A}=\mathbf{A}_{\textrm{sym}}+\mathbf{A}_{\textrm{skew}}$ where 
$\mathbf{A}_{\textrm{sym}}:=\tfrac12\left(\mathbf{A}+\mathbf{A}^\intercal\right)$ and
$\mathbf{A}_{\textrm{skew}}:=\tfrac12\left(\mathbf{A}-\mathbf{A}^\intercal\right)$, there holds
\begin{equation}\label{eq:integralidentity}
	\int_{\R} \W \cdot \left(\mathbf{A} \partial_x \W\right) \,\ud x
	= \int_{\R} \W \cdot \left(\mathbf{A}_{\textrm{skew}} \partial_x \W\right) \,\ud x.
\end{equation}
for any real-valued smooth function $\W$ such that $\W(\pm\infty)=0$,
Indeed for symmetric matrices $\mathbf{S}$, there holds
\begin{equation*}
	\int_{\R} \W \cdot \left(\mathbf{S} \partial_x \W\right) \,\ud x
	= \int_{\R} \left(\mathbf{S}^\intercal \W\right) \cdot \partial_x \W \,\ud x
	= \int_{\R} \left(\mathbf{S} \W\right) \cdot \partial_x \W \,\ud x
	= - \int_{\R} \left(\mathbf{S} \partial_x \W\right) \cdot \W \,\ud x
\end{equation*}
so that \eqref{eq:integralidentity} is zero for $\mathbf{A}$ symmetric, i.e. if $\mathbf{A}=\mathbf{A}_{\textrm{sym}}$.

Such property suggests the following preliminary definition.

\begin{definition}\label{def:parab}
System \eqref{eq:hyppar_cons} is said to be \emph{parabolic at  $\W_\ast$} if the (real) eigenvalues of the symmetric
matrix $\mathbf{D}_{\textrm{\rm sym}}:=\tfrac12\left(\mathbf{D}+\mathbf{D}^\intercal\right)$ lie in $(0,+\infty)$.
\end{definition}

In such a case, assuming appropriate boundary conditions at $\infty$ on $\W_\eps$,
it is possible to deduce an energy estimate for \eqref{eq:sys_lin}.
Precisely,  multiplying by $\W_\eps$ and integrating with respect to the space variable $x$, we end up with (after an additional integration by parts)
\begin{equation*}
	\frac{\ud}{\ud t}\left( \frac12 \|\W_\eps(t,\cdot)\|_{L^2}^2\right) +\eps \int_{\R} \partial_x \W_\eps\cdot \mathbf{D}\,\partial_x \W_\eps\,\ud x
		= -  \int_{\R} \W_\eps\cdot \left(\mathbf{A} \partial_x \W_\eps\right)\,\ud x 
\end{equation*}
which, taking into account \eqref{eq:integralidentity}, reduces to
\begin{equation*}
	\frac{\ud}{\ud t}\left( \frac12 \|\W_\eps(t,\cdot)\|_{L^2}^2\right) +\eps \int_{\R} \partial_x \W_\eps\cdot \mathbf{D}_{\textrm{sym}}\,\partial_x \W_\eps\,\ud x
		= -  \int_{\R} \W_\eps\cdot \mathbf{A}_{\textrm{skew}} \partial_x \W_\eps\,\ud x. 
\end{equation*}
For any $M>0$, the above equality provides the estimate
\begin{equation*}
	\begin{aligned}
	\frac{\ud}{\ud t}\left( \tfrac12 \|\W_\eps(t,\cdot)\|_{L^2}^2\right) &+\eps \int_{\R} \partial_x \W_\eps\cdot \mathbf{D}_{\textrm{sym}}\,\partial_x \W_\eps\,\ud x
		\leq C_{\mathbf{A}} \|\W_\eps(t,\cdot)\|_{L^2}\| \partial_x \W_\eps(t,\cdot)\|_{L^2}\\
		&\hskip2.65cm \leq \tfrac{1}{2}C_{\mathbf{A}} M^2\|\W_\eps(t,\cdot)\|_{L^2}^2+\tfrac{1}{2} C_{\mathbf{A}}M^{-2}\|\partial_x \W_\eps(t,\cdot)\|_{L^2}^2\,,
	\end{aligned}
\end{equation*}
with $C_{\mathbf{A}}$ depending only on $\mathbf{A}_{\textrm{skew}}$.
In particular, {\it if $\mathbf{A}$ is symmetric}, then $C_{\mathbf{A}}=0$
and parabolicity implies uniform stability.

In the general case, if system \eqref{eq:hyppar_cons} is parabolic, denoting by $\lambda_1>0$ the minimal eigenvalue of $\mathbf{D}_{\textrm{sym}}$,
we have ${\partial_x \W_\eps\cdot \mathbf{D}_{\textrm{sym}}\,\partial_x \W_\eps\geq \lambda_1\|\partial_x \W_\eps\|^2}$,
such that 
    \begin{equation*}
	\frac{\ud}{\ud t} \|\W_\eps(t,\cdot)\|_{L^2}^2
	+2\left(\eps \lambda_1-\tfrac{1}{2}C_{\mathbf{A}}M^{-2}\right)\|\partial_x \W_\eps(t,\cdot)\|_{L^2}^2
	\leq C_{\mathbf{A}} M^2\|\W_\eps(t,\cdot)\|_{L^2}^2\,.
      \end{equation*}
Then, choosing $M^2=C_{\mathbf{A}}/(2\eps \lambda_1)$, we infer the estimate
    \begin{equation*}
	\frac{\ud}{\ud t}\|\W_\eps(t,\cdot)\|_{L^2}^2 \leq  \frac{C_{\mathbf{A}}^2}{2\eps \lambda_1}\|\W_\eps(t,\cdot)\|_{L^2}^2\,.
      \end{equation*}
Hence, by a straightforward application of Gr\"onwall's Lemma, we infer the bound
\begin{equation*}
	\|\W_\eps(t,\cdot)\|_{L^2}\leq C_{\eps,T}\|\W_\eps(0,\cdot)\|_{L^2}
\end{equation*}
where $C_{\eps,T}=\exp\left\{C_{\mathbf{A}}^2T/(4\eps \lambda_1)\right\}$ tends to $+\infty$ as $\eps\to 0^+$
if $C_{\mathbf{A}}>0$. 
Hence, it is transparent that such bounds do not provide any information
relative to the (eventual) \emph{uniform} stability of system \eqref{eq:sys_lin}. In
fact, some choices of (non-symmetric) $\mathbf{A}$ lead to the non uniform
stability of \eqref{eq:sys_lin}.

Differently, let us explore the case in which there exists a symmetric positive definite matrix $\mathbf{X}$
such that $\mathbf{X}\mathbf{A}$ is symmetric and $(\mathbf{X}\mathbf{D})_{\textrm{sym}}$ is positive definite.
Then, multiplying the linear system in \eqref{eq:sys_lin} by $\mathbf{X}$, we obtain the modified system
\begin{equation}
	\mathbf{X}\,\partial_t \W_\eps + \mathbf{X}\mathbf{A} \partial_x \W_\eps
		=\eps \mathbf{X}\mathbf{D}\partial_x^2 \W_\eps\,.
\end{equation}
Next, let us proceed as before: multiplying by $\W_\eps$ and integrating over $\R$,
\begin{equation*}
	\frac{\ud}{\ud t}\|{\mathbf X}^{1/2}\W_\eps(t,\cdot)\|_{L^2}^2
	+2\eps \int_{\R} \partial_x \W_\eps\cdot (\mathbf{X}\mathbf{D})_{\textrm{sym}}\,\partial_x \W_\eps\,\ud x\leq 0
\end{equation*}
having used the identity \eqref{eq:integralidentity} to the symmetric matrix $\mathbf{X}\mathbf{A}$ which provides a corresponding starting
energy estimates for $\|{\mathbf X}^{1/2}\W_\eps\|_{L^2}$ which is also uniform with respect to $\eps$.
Uniform stability is thus guaranteed under the assumption of the existence of a symmetrizer $\mathbf{X}$ with the properties described above.

When the system of conservation laws \eqref{eq:hyppar_cons} possesses an entropy $\zeta$, it can be proved that  $\mathrm d^2\zeta$
is indeed a symmetrizer for \eqref{eq:hyppar_cons} and, thus, plays the role of $\mathbf{X}$ previously used to deduce
an energy estimate uniform in $\eps$.
Entropy and its compatibility with the diffusion matrix thus allow us to derive  stability estimates that are stronger 
than the ones obtained by using the parabolicity of the matrix $\mathbf D$.
This issue will be further illustrated later on.

If the matrix $(\mathbf{X}\mathbf{D})_{\textrm{sym}}$ is only positive semi-definite, additional assumptions are required.
Among others, a well-established approach posits that the celebrated {\it Kawashima--Shizuta condition} holds,
consisting in the request that the linear equation in \eqref{eq:sys_lin} is such that {\it no eigenvectors of $\mathbf{A}$ are in the kernel of
$\mathbf{D}$}  (see \cite{Kaw, KS}). 
Difficulties relative to the case in which the above condition is not satisfied are explored in details in \cite{BeauZuaz11, MascNata10}.

\section{Flowing regime for the Burgers fluid-particle system}
\label{sec:Bur}

Let us assume that the  carrier  fluid  is incompressible in the sense that $n_\eps\equiv 1$ in $(0,\infty)\times\mathbb{R}$,
so that the dimensionless hybrid density of the mixture becomes  $r=1+\rho$.
Incidentally, let us observe that this is not the standard incompressibility assumption required in fluid-dynamics
giving rise to Euler and Navier-Stokes equations for incompressible media.
Indeed, assuming that the carrier fluid keeps a constant homogeneous density  is a quite crude assumption.
Even if controversial in principle, it makes some computations easier and more explicit, allowing to
bring out interesting structural properties of the model.
It is worth pointing out the analysis of traveling wave solutions and their stability has been already performed
in \cite{DR} for a variant of this toy-model with temperature $\theta=0$ and non-zero fluid viscosity.

\subsection{Derivation and hyperbolicity}
Given $\theta, \eps>0$, let us consider the coupled fluid-kinetic system
\begin{equation}\label{eq:cp2s}
	\left\{\begin{aligned}
	\partial_t f _\eps+v\partial_x f_\eps&=\eps^{-1}\partial_v\left\{(v- u_\eps)f_\eps+\theta\partial_v f_\eps\right\},\\
	\partial_t u_\eps + \partial_x u_\eps^2 &=\eps^{-1}(J_\eps-\rho_\eps u_\eps),
	\end{aligned}\right.
\end{equation}
where
\begin{equation*}
	\rho_\eps(t,x)=\int f_\eps(t,x,v)\ud v
		\quad\textrm{and}\quad
	J_\eps(t,x)=\int v  f_\eps(t,x,v)\ud v,
\end{equation*}
As explained in the Introduction, the expected limit as $\eps\to 0$ is  system \eqref{eq:red}.

\begin{rmk}\label{rmk:invariance}\rm 
As stated before, system \eqref{eq:red} is \emph{not} invariant under Galilean transformations.
Indeed, let us consider the change of variables $(s,y)=(t,x-u_0t)$, with $u_0\in \mathbb R$ a constant velocity, corresponding to
$(\partial_t, \partial_x)=(\partial_s-u_0\partial_y,\partial_y)$ and set $v:=u-u_0$.
Applying the transformation to the first equation in \eqref{eq:red}, we infer
\begin{equation*}
	\partial_t \rho +\partial_x(\rho u) = \partial_s \rho -u_0\partial_y \rho + \partial_y\bigl\{\rho(v+u_0)\bigr\}
		= \partial_s \rho +\partial_y(\rho v).
\end{equation*}
Concerning the second equation, we deduce upon computation
\begin{equation*}
	\begin{aligned}
	\partial_t (ru) +\partial_x (ru^2+\theta \rho)
		&= \partial_s (rv)+\partial_y (rv^2+\theta\rho)+u_0\partial_y v.
	\end{aligned}
\end{equation*}
In particular, in the new reference frame $(s,y)$, system \eqref{eq:red} becomes
\begin{equation*}
	\left\{\begin{aligned}
	&\partial_s \rho +\partial_y(\rho v) = 0,\\
	&\partial_s (rv)+\partial_y (rv^2+\theta \rho)+u_0 \partial_y v= 0,
	\end{aligned}\right.
\end{equation*}
with $v:=u-u_0$, coinciding with the previous system if and only if $u_0=0$.

Differently, system \eqref{eq:red} is invariant under space reversal:
indeed, applying the transformation $(s,y)=(t,-x)$ and $v=-u$, we obtain
\begin{equation*}
	\left\{\begin{aligned}
	&\partial_t \rho +\partial_x(\rho u) = \partial_s \rho - \partial_y(-\rho v) = \partial_s \rho +\partial_y(\rho v) = 0\,,\\
	&\partial_t (ru)  +\partial_x (ru^2+\theta \rho)= -\partial_s (rv) -\partial_y(rv^2+\theta \rho)= 0.
	\end{aligned}\right.
\end{equation*}
\end{rmk}

System \eqref{eq:red} can be cast in a conservative vector form \eqref{eq:hyp_cons} where
\begin{equation}\label{eq:WandF_Bfp}
	\W=(\rho,w)^\intercal\qquad\textrm{and}\qquad F( \W)=(\rho w/r,w^2/r +\theta\rho)^\intercal\,,
\end{equation}
where $w=r u$.
Examining hyperbolicity amounts to focus on the linearization
\begin{equation*}
	\partial_t\W+\dd F(\W_\ast)\partial_x \W=0,
\end{equation*}
where
\begin{equation*}
	\dd F(\W):=\begin{pmatrix} w/r^2 & \rho/r \\ -w^2/r^2+\theta & 2w/r \end{pmatrix}
		=\begin{pmatrix} u/r & \rho/r \\ -u^2+\theta & 2u \end{pmatrix}.
\end{equation*}
In the following computations, let us drop the subscript $\ast$ for the sake of shortness.            
By definition, system \eqref{eq:hyp_cons}  is {\it strictly hyperbolic} at $\W$ if and only if the polynomial
\begin{equation*}
	p(\lambda):=\det\bigl(\dd F(\W)- \lambda\mathbf{I}\bigr)=0
\end{equation*}
has distinct real roots.
Upon substitution, we obtain
\begin{equation*}
\lambda^2-2\left(1+\frac{1}{2r}\right)u\lambda+\frac{2+\rho}{r}\,u^2-\frac{\theta\rho}{r}=0
\end{equation*}
whose solutions are
\begin{equation}\label{eq:evalues}
 	\lambda_\pm(\W):=u\pm\frac{{\sqrt{u^2+\theta\delta^2}\pm u}}{2r}
	\qquad\textrm{with}\quad \delta(\rho):=2\sqrt{\rho\,r}\,.
\end{equation}
Given $\theta>0$, the function $\rho\mapsto \delta(\rho)$ is invertible for $\rho\in [0,+\infty)$.
Indeed, the relation defining $\chi:=\delta^2=4\rho\,r=4\rho(1+\rho)$ can be rewritten
as a second order polynomial in $\rho$, viz. $4\rho^2+4\rho-\chi=0$.
Taking the positive root in the standard formula for the roots of second order polynomials, we infer
\begin{equation*}
	\rho=\varphi(\chi):=\frac{\sqrt{1+\chi}-1}{2}
		=\frac{1}{2}\ \frac{\chi}{\sqrt{1+\chi}+1}.
\end{equation*}
If $\rho$ is strictly positive, so are $\delta$ and $\chi$, thus the system is strictly hyperbolic for $\theta>0$.

To classify the type of hyperbolic system we are dealing with, we analyse
the scalar product $\nabla_{\W}\lambda_\pm \cdot \mathbf{r}_\pm$ where $\mathbf{r}_\pm$
are right eigenvectors of the matrix $\dd F- \lambda\mathbf{I}$ relative to $\lambda_\pm$.  

\begin{prop}\label{prop:redfields}
For $\theta>0$, system \eqref{eq:red} is strictly hyperbolic with two genuinely
nonlinear fields for $(\rho,ru)\in(0,\infty)\times\mathbb{R}$.
\end{prop}
  
\begin{proof}
System \eqref{eq:hyp_cons} can be also written as a system in $\U:=(\rho,u)$:
\begin{equation}\label{eq:hyp_noncons}
	\partial_t G(\U)+\partial_x H(\U)=0
\end{equation}
where the functions $G(\U)=(\rho,ru)^\intercal$ and $H(\U)=(\rho
u,ru^2+\theta\rho)^\intercal$ are such that
\begin{equation*}
	\dd G(\U):=\begin{pmatrix} 1 & 0 \\  u & r \end{pmatrix},\qquad 
	\dd H(\U):=\begin{pmatrix}  u & \rho \\ u^2+\theta & 2ru \end{pmatrix}.
\end{equation*}
Let us set $\mu_{\pm}(\U)=\lambda_{\pm}(G(\U))$. 
In particular, $\mu_\pm\bigr|_{u=0}=\pm \sqrt{\theta\rho/r}$.
By Lemma~\ref{lem:lemm_gal}, it is equivalent to compute $\nabla_{\U}\mu_\pm \cdot \mathbf{s}_\pm$
where  $(\dd H -\mu_\pm \dd G)\mathbf{s}_\pm=0$.
In turn, this reduces to finding $\mathbf{s}_\pm$ such that $(u-\mu_\pm,\rho)\cdot\mathbf{s}_\pm=0$.
Let us choose $\mathbf{s}_\pm=\left(\rho,\mu_\pm-u\right)^{\intercal}$,
so that the functions $\U\mapsto \mathbf{s}_\pm(\U)$ are smooth on $(0+\infty)\times\R$.

The auxiliary function $\sigma\,:\,\mathbb{R}\to(-1,1)$, defined by $\sigma(x):={x}/{\sqrt{1+x^2}}$, see Fig.~\ref{fig:estimate},
is continuous, odd and such that
\begin{equation}\label{eq:sigmaprop}
	0\leq |\sigma(x)|\leq \min\left\{1, |x|\right\},\qquad \sigma'(x)=(1+x^2)^{-3/2}\,.
\end{equation}
\begin{figure}[!htb]
\includegraphics[width=10cm]{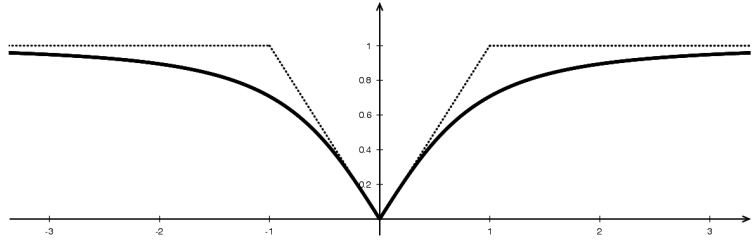}
\caption{\footnotesize    Graph of the function $x\mapsto |\sigma(x)|$ (continuous line)
compared to the one of $x\mapsto \min\left\{1, |x|\right\}$ (dotted line) for $x\in\mathbb{R}$.}
\label{fig:estimate}
\end{figure}

Moreover, $\sigma$ is invertible with inverse $\psi\,:\,(-1,1)\to\mathbb{R}$ given by $x=\psi(y):={y}/{\sqrt{1-y^2}}$.
In term of $\sigma$, the eigenvalues $\mu_\pm$ can be represented as
\begin{equation*}
	\mu_\pm(\U)=u\pm\frac{1}{2r}(1\pm \sigma)\sqrt{u^2+\theta\delta^2}
\end{equation*}
with $\sigma$ computed at $u/\sqrt{\theta\delta^2}$.

Since the gradient $\nabla_{\U} \mu_\pm=\left(\partial_\rho \mu_\pm, \partial_ u \mu_\pm\right)$ is given by
\begin{equation*}
	\partial_\rho \mu_\pm  (\U)=-\frac{(1\pm \sigma)u}{2r^2}\pm \frac{\theta}{r\sqrt{u^2+\theta\delta^2}}\,,\qquad
	\partial_ u \mu_\pm (\U)=1+\frac{1\pm \sigma}{2r}\,,
\end{equation*}
there holds
\begin{equation*} 
	\begin{aligned}
	\nabla_{\U} \mu_\pm (\U)&\cdot \mathbf{s}_\pm
	=\left(-\frac{(1\pm \sigma)u}{2r^2}\pm \frac{\theta}{r\sqrt{u^2+\theta\delta^2}},1+\frac{1\pm \sigma}{2r}\right)
		\cdot\left(\rho,\pm\frac{1\pm \sigma}{2r}\sqrt{u^2+\theta\delta^2}\right)\\
	&=-\frac{(1\pm \sigma)\rho u}{2r^2}\pm \frac{\theta\rho }{r\sqrt{u^2+\theta\delta^2}}
		\pm \frac{1\pm \sigma}{2r}\sqrt{u^2+\theta\delta^2}\pm\frac{(1\pm\sigma)^2}{4r^2}\sqrt{u^2+\theta\delta^2}\\
	&=\pm \frac{1\pm \sigma}{2r}\left\{\sqrt{u^2+\theta \delta^2}+ \frac{(1\pm \sigma)\sqrt{u^2+\theta \delta^2}}{2r}
		\mp \frac{\rho u }{r}\right\}\pm \frac{\theta\rho}{r\sqrt{u^2+\theta \delta^2}}.
	\end{aligned}
\end{equation*}
Since $r=1+\rho$ and $\sigma=u/\sqrt{u^2+\theta \delta^2}$, the three terms in braces can be recast as
\begin{equation*}
	\begin{aligned}
	\sqrt{u^2+\theta\delta^2}&+\frac{(1\pm \sigma)\sqrt{u^2+\theta \delta^2}}{2r}\mp\frac{\rho u}{r}
		=\left(1+\frac{1\pm \sigma}{2r}\mp\frac{\rho \sigma}{r}\right)\sqrt{u^2+\theta\delta^2}\\
	&=\left\{1+\frac{1\pm \sigma}{2}+\rho(1\mp\sigma)\right\}\frac{\sqrt{u^2+\theta\delta^2}}{r}
		\geq \frac{\sqrt{u^2+\theta\delta^2}}{r}\geq 0\,,
	\end{aligned}
\end{equation*}
with the equality holding only for $\U=\mathbf{0}$ in the case $\theta>0$.
Hence, for $\rho>0$, there hold
\begin{equation*}
	\nabla_{\W} \lambda_-\cdot \mathbf{r}_-=\nabla_{\U} \mu_-\cdot \mathbf{s}_-<
	0<\nabla_{\U} \mu_+\cdot \mathbf{s}_+=\nabla_{\W} \lambda_+\cdot \mathbf{r}_+\,,
\end{equation*}
 where we make use of Lemma~\ref{lem:lemm_gal}.
\end{proof}

\subsection{Shock solutions}
Shock waves of system \eqref{eq:hyp_cons} are special solutions $\W(x,t)=\mathrm W(y)$
depending only on the variable $y: =x-ct$ with the form of a pure jump
\begin{equation*}
	\W(x,t)=\mathrm W(y):=\left\{\begin{aligned}
		&\W_\ast	\quad &\textrm{if}\quad 	&y<0,\\
		&{\W}	\quad &\textrm{if}\quad	&y\geq 0.
		\end{aligned}\right. 
\end{equation*}
where $\W_\ast:=(\rho_\ast, r_\ast u_\ast)$ and $\W:=(\rho, ru)$.
In presence of Galilean invariance, we could focus without loss of generality on stationary solutions $\W$, i.e. $c=0$ and $y=x$.
Unfortunately, as observed in Remark \ref{rmk:invariance}, system \eqref{eq:red} does not possess such a symmetry
and the corresponding reduction cannot be considered.

In order to be weak solutions, such functions are forced to satisfy the Rankine--Hugoniot conditions \eqref{eq:gen_rh}.
For system \eqref{eq:red}, they 
 take the specific form 
\begin{equation}\label{eq:RH0}
	\left\{\begin{aligned}
	&-c \llbracket \rho \rrbracket + \llbracket \rho u \rrbracket=0,\\
	&-c \llbracket ru \rrbracket + \llbracket ru^2 + \theta \rho \rrbracket=0,
	\end{aligned}\right.
    \end{equation}
where  $\llbracket g \rrbracket=g-g_\ast$.

Given $\rho_\ast$ and $u_\ast$, let us show that these relations lead to $u$ being a function of $\rho$.
 If $\llbracket \rho\rrbracket =0$, then from the first equation in \eqref{eq:RH0},
we infer $\rho_\ast\llbracket u\rrbracket =0$. 
Hence, assuming $\rho_\ast>0$, we are forced to have $\llbracket u\rrbracket =0$, so that the solution is actually a constant state.
Being interested in non constant profiles, we assume $\llbracket \rho\rrbracket \neq 0$.
Then, the propagation speed can be expressed as
\begin{equation}\label{eq:explicitvelocity0}
	c= \frac{\llbracket \rho  u \rrbracket}{\llbracket \rho \rrbracket}.
\end{equation}
Next, we are going to use the two following relations, valid for any functions $f$ and $g$,
\begin{equation}\label{eq:jumpuseful}
 	\llbracket f g \rrbracket = \llbracket f  \rrbracket  g_\ast + f \llbracket   g \rrbracket\qquad\textrm{and}\qquad
	\llbracket f g^2\rrbracket =  \llbracket f \rrbracket g_\ast^2 + 2  f g_\ast \llbracket  g \rrbracket + f \llbracket  g \rrbracket^2\,.
\end{equation}
Substituting \eqref{eq:explicitvelocity0} in the identity \eqref{eq:RH0}, we obtain
\begin{equation*}
	 \llbracket \rho  u \rrbracket \llbracket  u \rrbracket + \llbracket \rho  u \rrbracket^2
		= \llbracket \rho  u \rrbracket \llbracket r  u \rrbracket 
		= \llbracket \rho \rrbracket \llbracket r u^2 \rrbracket  + \theta \llbracket \rho \rrbracket^2\,,
\end{equation*}
and, taking advantage of \eqref{eq:jumpuseful}, we infer
\begin{equation*}
	\rho_\ast r \llbracket  u \rrbracket^2 - \llbracket \rho \rrbracket  u_\ast  \llbracket  u \rrbracket - \theta \llbracket \rho \rrbracket^2=0.
\end{equation*}
Considering the form \eqref{eq:hyp_noncons} of the original system \eqref{eq:hyp_cons}, the set of admissible shocks
$\mathcal{H}_{\W_\ast}$ of a given state $\W_\ast=(\rho_\ast,r_\ast u_\ast)$, usually called {\it Hugoniot locus}, is  given by the union
of two distinct branches, here denoted by $\mathcal{H}_{\W_\ast,+}$ and $\mathcal{H}_{\W_\ast,-}$ (see Figure~\ref{f1}) 
\begin{equation}\label{eq:defus}
	\mathcal{H}_{\W_\ast,\pm}  = \left\{(\rho,ru_\pm)\,:\,\rho>0,u_\pm(\rho)
		=u_\ast\pm \dfrac{\llbracket \rho \rrbracket}{\rho_\ast}\cdot\dfrac{\sqrt{u_\ast^2+\theta \Delta^2}\pm u_\ast}{2r}\right\},
\end{equation}  
with $\Delta(\rho,\rho_\ast):=2\sqrt{\rho_\ast r}$.
\begin{figure}[!htb]
\includegraphics[height=8cm]{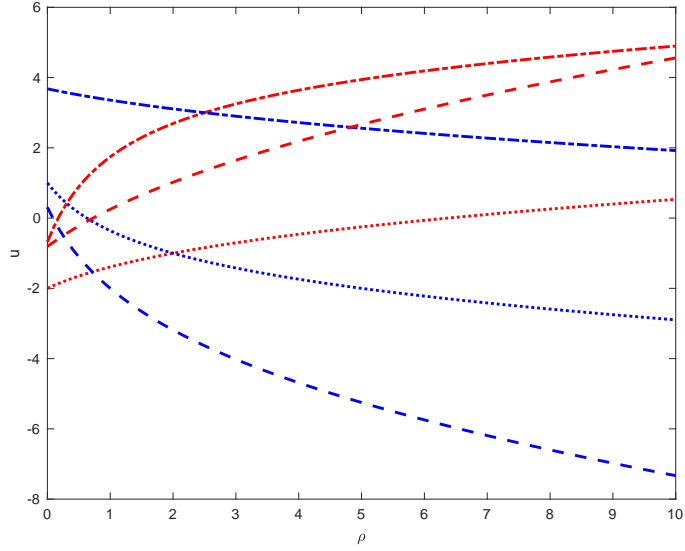}
\caption{\footnotesize  Hugoniot locus for several states $\U_\ast=(\rho_\ast, u_\ast)$.
Plots are given of several curves $\rho\mapsto u=u(\rho)$ defined by \eqref{eq:defus},  $\U_\ast$ being the intersection
point of the two curves drawn with the same line-specification (dotted, dashed or dot-dashed).}
\label{f1}
\end{figure}
Accordingly, along each branch, the shock speed is given by \eqref{eq:explicitvelocity0}, that becomes, using again \eqref{eq:jumpuseful},
\begin{equation}\label{eq:defcs}
	\begin{aligned}
	c_\pm(\rho;\W_\ast)=  u_\ast+\rho\frac{\llbracket  u \rrbracket}{\llbracket \rho \rrbracket} 
	= u_\ast \pm \frac{\rho}{\rho_\ast}\cdot\frac{\sqrt{ u_\ast^2 +\theta\Delta^2}\pm  u_\ast}{2r}.
	\end{aligned}
\end{equation}
Note that we can equally write 
 \begin{equation*}
	\begin{aligned}
	c_\pm(\rho;\W_\ast)= u+\rho_\ast\frac{ \llbracket  u \rrbracket}{\llbracket \rho \rrbracket}
	= u \pm \frac{\sqrt{u_\ast^2 +\theta\Delta^2}\pm u_\ast}{2r}.
	\end{aligned}
\end{equation*}
With the sign $(+)$, respectively $(-)$, $c_\pm$ is larger, resp. smaller, than both the left velocity $u_\ast$ and the right velocity $u$.

Specifically, we regard at the curves defined by \eqref{eq:defus} and \eqref{eq:defcs} as parametrizations of the states $\W$ that
can be connected to $\W_\ast$ by a pure discontinuity providing a weak solution to system \eqref{eq:hyp_cons} with corresponding
parameter given by the density  $\rho\in (0,+\infty)$.
As a matter of fact, we observe that \eqref{eq:defcs} satisfies
\begin{equation*}
	c_\pm(\rho_\ast;\W_\ast):=\lim_{\rho\to \rho_\ast} c_\pm(\rho;\W_\ast)=\lambda_\pm(\W_\ast)\quad\textrm{for}\;\;\rho_\ast>0
	\qquad\textrm{and}\qquad c_\pm(0;\W_\ast)= u_\ast\,.
\end{equation*}
Moreover, since
\begin{equation*}
	\begin{aligned}
 	\partial_\rho c_\pm (\rho;\W_\ast)
        &=\pm\frac{1}{2\rho_\ast r}\left\{\frac{\sqrt{u_\ast^2 +\theta\Delta^2}\pm u_\ast}{r}
            +\frac{2\theta \rho_\ast\rho}{\sqrt{u_\ast^2 +\theta\Delta^2} }\right\}.
	\end{aligned}
\end{equation*}
there hold, for $\rho_\ast>0$,
\begin{equation*}
	\partial_\rho c_-(\rho;\W_\ast)<0<\partial_\rho c_+(\rho;\W_\ast).
\end{equation*}
As a consequence, we infer the equivalences valid for any $\rho$ between $\rho_\ast$ and $\bar \rho$
\begin{equation}\label{eq:liu}
	\begin{aligned}
	&c_+(\bar \rho;\W_\ast)-c_+(\rho;\W_\ast)=\int_\rho^{\bar \rho} \partial_\rho c_+(\xi;\W_\ast)\,\dd \xi<0
		\quad\iff\quad 0\leq \bar \rho<\rho\,,\\
	&c_-(\bar \rho;\W_\ast)-c_-(\rho;\W_\ast)=\int_\rho^{\bar \rho} \partial_\rho c_-(\xi;\W_\ast)\,\dd \xi<0
		\quad\iff\quad 0\leq \rho<\bar \rho\,.
	\end{aligned}
\end{equation}
In particular, the (strict) Liu's entropy criterion is satisfied for $\bar\rho<\rho_\ast$
in the case of sign $+$ and for $\rho_\ast<\bar\rho$ in the case of sign $-$ (see \cite{Liu}).
Since the system is genuinely nonlinear, this is equivalent to Lax's entropy condition for weak shocks \cite[Theorem~8.4.2]{Daf}.
 
\subsection{Entropy for the inviscid Burgers fluid-particle system}\label{sec:ent:fkm}

The quantity
\begin{equation}\label{eq:entropyB}
	\mathscr{H}(f,u):=\underbrace{\dfrac{1}{2}u^2}_{fluid}+\underbrace{\int H(f,v)\ud v}_{particles}
	\quad\textrm{where}\;\; H(f,v):=f\bigl(\tfrac{1}{2}v^2 + \theta \ln f\bigr)
\end{equation}
defines an entropy for the fluid-kinetic model \eqref{eq:cp2s}.
Indeed, as previously observed, the kinetic equation in \eqref{eq:cp2s} can be rephrased as
\begin{equation*}
	\partial_t f _\eps+v\partial_x f_\eps=\dfrac{\theta}{\eps}\partial_v\left\{M_{u_\eps}\partial_v(M_{u_\eps}^{-1}f_\eps)\right\}\,,
\end{equation*}
which involves the maxwellian $M_u$ defined in \eqref{eq:maxwell}, to be considered coupled with
\begin{equation*}
	\begin{aligned}
	\partial_t u_\eps + \partial_x u_\eps^2
	=\dfrac{\theta}{\eps} \int \left\{M_{u_\eps}\partial_v(M_{u_\eps}^{-1}f_\eps)-\partial_vf_\eps\right\}\ud v
	=\dfrac{\theta}{\eps} \int M_{u_\eps}\partial_v(M_{u_\eps}^{-1}f_\eps)\ud v.
 	\end{aligned}
\end{equation*}
since $f_\eps$ is assumed to vanish at $\infty$.
Next, setting 
\begin{equation*}
	\mathscr{G}(f,u):=\tfrac{2}{3}u^3+\int vf\left(\tfrac{1}{2}v^2+\theta\ln f\right)\ud v\,,
\end{equation*}
we infer, integrating by parts,
\begin{equation*}
	\begin{aligned}
  	\partial_t \mathscr{H}&
	=\int \left\{\tfrac{1}{2}v^2+\theta(\ln f_\eps+1)\right\}\partial_t f_\eps \ud v+u_\eps\partial_t u_\eps\\
	&=-\partial_x\mathscr{G}+\dfrac{\theta}{\eps} \int u_\eps M_{u_\eps}\partial_v(M_{u_\eps}^{-1}f_\eps)\ud v\\
	&\quad +\dfrac{\theta}{\eps}\int \left\{\tfrac{1}{2}v^2+\theta(\ln f_\eps+1)\right\}\partial_v\left[M_{u_\eps}\partial_v(M_{u_\eps}^{-1}f_\eps)\right]\ud v\\
  	&=-\partial_x\mathscr{G}
		+\dfrac{\theta}{\eps}\int \partial_v(M_{u_\eps}^{-1}f_\eps)\left\{M_{u_\eps}(u_\eps-v)-\theta M_{u_\eps} f_\eps^{-1} \partial_v f_\eps\right\}\ud v\\
  	&=-\partial_x\mathscr{G}
		-\dfrac{\theta^2}{\eps} \int f_\eps^{-1}\partial_v(M_{u_\eps}^{-1}f_\eps)\left\{M_{u_\eps}(\partial_v f_\eps)-(\partial_v M_{u_\eps})f_\eps\right\}\ud v\,,
	\end{aligned}
\end{equation*}
since, as previously seen, $M_u(u-v)=\theta \partial_v M_u$.
Therefore, we deduce the estimate 
\begin{equation*}
	\partial_t \mathscr{H}+\partial_x \mathscr{G}
		=-\dfrac{\theta^2}{\eps}\int M_{u_\eps}^2f_\eps^{-1}\left\{\partial_v(M_{u_\eps}^{-1}f_\eps)\right\}^2\ud v\leq 0\,.
\end{equation*}
Next, let us focus on the regime $\eps\to 0^+$ for which the formal ansatz \eqref{eq:ansatz} is assumed to hold.
As a consequence, inspired by the kinetic representation of conservation laws \cite{Per}, we guess an entropy
for the limit system by evaluating the functional $\mathscr{H}$ at the equilibrium $\rho_\eps M_{u_\eps}$.

Preliminarly, let us observe that, knowing that
\begin{equation}\label{eq:moments}
	\int M_u\ud v=1,\qquad   \int (v-u) M_u\ud v=0,\qquad \int |v-u|^2 M_u\ud v =\theta\,,
\end{equation}
there holds
\begin{equation*}
	\int  v^2 M_{u} \ud v = \int  \left\{u^2+2u(v-u)+|v-u|^2\right\} M_{u} \ud v = u^2+\theta
\end{equation*}
Hence, inspired by the kinetic representation of conservation laws \cite{Per}, the formal identity
\begin{equation*}	
	\begin{aligned}
	\mathscr{H}(f_\eps,u_\eps)&\simeq \mathscr{H}(\rho_\eps M_{u_\eps},u_\eps)
		=\tfrac{1}{2}u_\eps^2+\int \rho_\eps M_{u_\eps} \left\{\tfrac{1}{2}v^2 + \theta \ln(\rho_\eps M_{u_\eps})\right\} \ud v\\
	&=\tfrac{1}{2}r_\eps u_\eps^2+\tfrac12\theta\rho_\eps +\rho_\eps \int M_{u_\eps} \left\{\theta \ln \rho_\eps
		-\tfrac12\theta \ln(2\pi\theta)-\tfrac{1}{2}(v-u)^2\right\} \ud v\\
	&=\tfrac12r_\eps u_\eps^2 +\theta\rho\left\{\ln\rho_\eps-\tfrac12\ln(2\pi\theta)\right\}
	\end{aligned}
\end{equation*}
suggests the (simplified) choice $\eta(\U)=\tfrac12ru^2 +\theta\rho\ln\rho$, obtained by disregarding the linear term in $\rho$
(since we already know that $\rho$ satisfies a convection equation), with corresponding entropy flux given
by $q(\U)=\bigl(\tfrac 23+\tfrac12\rho\bigr)u^3+\theta\rho(\ln\rho+1)u$.
The pair $(\eta,q)$ is an entropy/entropy flux pair for \eqref{eq:hyp_noncons}.
Indeed, let us set 
\begin{equation*}
	Q:=\partial_t (\rho M_u)+v\partial_x( \rho M_u).
\end{equation*}
Using again \eqref{eq:moments},  we infer for any (smooth) solution $(\rho, u)$ of \eqref{eq:hyp_noncons}
\begin{equation*}
	\int \begin{pmatrix} 1\\ v \end{pmatrix} Q\ud v= -
	\begin{pmatrix} 0 \\ \partial_ t u+\partial_x u^2 \end{pmatrix}
\end{equation*}
since integration with respect to $v$ yields the system of conservation laws.
It follows that
\begin{equation*}
	\begin{aligned}
	\partial_t \eta+\partial_x q
	&= \partial_t \left(\tfrac{1}{2}u^2\right)+\partial_x \left(\tfrac{2}{3}u^3\right)
		+\int Q\left\{\tfrac{1}{2}\,v^2+\theta\ln(\rho M_u)-\tfrac12\theta\ln(2\pi\theta)+1\right\}\ud v\\
	&=\partial_t \left(\tfrac{1}{2}u^2\right)+\partial_x \left(\tfrac{2}{3}u^3\right)+\tfrac12\int Q\left(v^2-|v-u|^2\right)\ud v\\
	&=\partial_t \left(\tfrac{1}{2}u^2\right)+\partial_x \left(\tfrac{2}{3}u^3\right)+u\int v\,Q\ud v 
		=0.
\end{aligned}
\end{equation*}
In terms of the  variables $\W=(\rho,w)$, the entropy $\zeta$ is given by
\begin{equation}\label{eq:entropyBfp}
	\zeta(\W)=\frac{w^2}{2r}+\theta\rho\ln \rho.
\end{equation}
Upon differentiation, denoting by the same symbols $\nabla_{\W}\zeta$ and $\mathrm d^2_{\W}\zeta$
the corresponding vector/matrix computed both at $\W$, we obtain
 the following expressions that will be useful later on
\begin{equation}\label{eq:ent2_2}
	\begin{aligned}
	\nabla_{\W}\zeta (\W) ^\intercal &=\left(-{w^2}/(2r^2)+\theta(1+\ln\rho),{w}/{r}\right)
		=\left(-u^2/2+\theta(1+\ln\rho),u\right)\,,\\
	\DD^2_{\W}\zeta (\W)
	&=\begin{pmatrix} 	{w^2}/{r^3}+{\theta}/{\rho} 	& -{w}/{r^2}	\\
					-{w}/{r^2}		 		& {1}/{r}		\end{pmatrix}
	=\begin{pmatrix}	{u^2}/{r}+{\theta}/{\rho} 	& -{u}/{r}		\\
					-{u}/{r}			 	& {1}/{r}		\end{pmatrix}.
	\end{aligned}
\end{equation}
In addition, $\zeta$ is a convex function, since the hessian $\DD^2_{\W}\zeta$ is positive definite. 

The function $\zeta$ defined in \eqref{eq:entropyBfp} furnishes an entropy for system \eqref{eq:hyp_cons}.
Hence, the matrix $\mathbf{X}:=\DD^2_{\W}\zeta$ is a symmetrizer for the flux $F$ as can be directly checked
(in fact, such property holds true for general hyperbolic systems, see \cite{Daf, Mock80}).

\subsection{Viscous corrections leading to \eqref{eq:bfp}}

We now use the Chapman-Enskog expansion to get the diffusive correction associated to system \eqref{eq:red}.
Specifically, we search for a hydrodynamic model with an appropriate modification, namely $(\rho_\eps,u_\eps)$
(where the dependence on $\eps$ is explicitly stated) satisfies
\begin{equation*}
	\partial_t \begin{pmatrix} \rho_\eps \\ r_\eps u_\eps\end{pmatrix}
	+\partial_x\begin{pmatrix} \rho_\eps u_\eps \\ r_\eps u_\eps^2 +\theta \rho_\eps \end{pmatrix}
	=\mathscr O(\eps).
\end{equation*}
In order to define the correction term, we expand the solution of the kinetic equation as
\begin{equation*}
	f_\eps=\rho_\eps M_{u_\eps} +\eps g_\eps,
	\qquad \int f_\eps\ud v=\rho_\eps,
	\qquad \int g_\eps\ud v=0,
\end{equation*}
where $M_u$ is the Maxwellian distribution defined in \eqref{eq:maxwell}.
Recalling the identity \eqref{eq:fp_form}, the system can be rewritten as
\begin{equation*}
	( \partial_t+v\partial_x)(\rho_\eps M_{u_\eps}+\eps g_\eps) = L_{u_\eps}(g_\eps),
\end{equation*}
coupled with the equation for $u_\eps$
\begin{equation*}
	\partial_t u_\eps + \partial_x u_\eps^2 =  \int (v-u_\eps) g_\eps\ud v.
\end{equation*}
Note that the integration of the kinetic equation yields
\begin{equation}\label{eq:hyd1_1}
	\partial_t \rho_\eps + \partial_x (\rho _\eps u_\eps) +\eps \partial_x\int (v-u_\eps)g_\eps\ud v=0,
\end{equation}
and
\begin{equation}\label{eq:hyd1_2}
	\begin{aligned}
	\partial_t \left(\rho_\eps u_\eps+ \eps \int vg_\eps\ud v\right) + \partial_x \left(\rho_\eps u_\eps^2+\theta\rho_\eps+ \eps \int v^2 g\ud v \right)
	 	&=-\int (v-u_\eps) g_\eps\ud v.\\
 		&= -\partial_t u_\eps - \partial_x u_\eps^2.
 	\end{aligned}
\end{equation}
We compute 
\begin{equation*}
	\begin{aligned}
 	(\partial_t+v\partial_x)(\rho_\eps M_{u_\eps})
	&= M_{u_\eps}\bigl\{\partial_t\rho_\eps +\partial_x (\rho_\eps u_\eps)\bigr\}
		+ (v-u_\eps)M_{u_\eps}\partial_x \rho_\eps - M_{u_\eps}\rho_\eps\partial_x u_\eps \\
 	&\quad +\dfrac{(v-u_\eps)}{\theta} \rho_\eps  M_{u_\eps}(\partial_t u _\eps +u_\eps\partial_x  u_\eps)
 		+ \rho_\eps  M_{u_\eps} \dfrac{ |v-u_\eps|^2}{\theta}\partial_x  u_\eps\\
 	&=\dfrac{(v-u_\eps)}{\theta} \rho_\eps M_{u_\eps}\left\{\int (v-u_\eps) g_\eps\ud v
		+ \theta\dfrac{1}{\rho_\eps} \partial_x \rho_\eps -   u_\eps\partial_x u_\eps\right\}\\
 	&\quad +\rho_\eps  M_{u_\eps}\left(\dfrac{|v-u_\eps|^2}{\theta}-1\right)\partial_x  u_\eps \\
	& = L_{u_\eps}(g_\eps) - \eps\left(\partial_t+v\partial_x\right)g_\eps.
 	\end{aligned}
\end{equation*}
From now on, we neglect the last $\mathscr
O(\eps)$ terms and thus obtain a relation that defines $g_\eps$ by inverting
$L_{u_\eps}$ as we are going to detail now.
Multiplying and integrating over $v$, we find
\begin{equation}\label{eq:intvg}
  	\int v g_\eps\ud v=\int (v-u_\eps) g_\eps\ud v
	= \dfrac{1}{r_\eps}\left(\rho_\eps u_\eps\partial_x u_\eps-\theta\partial_x \rho_\eps\right).
\end{equation}
 Hence, we are led to 
\begin{equation*}
 	L_{u_\eps}(g_\eps)=\dfrac{\theta(v-u_\eps)M_{u_\eps}}{r_\eps}\left(\theta \partial_x\rho_\eps- \rho_\eps u_\eps \partial_x u_\eps\right)
 		+ \rho_\eps  M_{u_\eps}\left(dfrac{|v-u_\eps|^2}{\theta}-1\right)\partial_x  u_\eps.
\end{equation*}
Observe that the integral with respect to $v$ of all terms in the right-hand side vanishes.   
Bearing in mind that 
\begin{equation*}
	L_0(vM_0)=-vM_0 \qquad\textrm{and}\qquad
	L_0\left(\left(\dfrac{v^2}{\theta} -1\right)M_0\right)=-2\left(\dfrac{v^2}{\theta} -1\right)M_0,
\end{equation*}
we obtain 
\begin{equation*}
  	g_\eps=-\dfrac12\left(\dfrac{|v-u_\eps|^2}{\theta}-1\right) \rho_\eps M_{u_\eps} \partial_x  u_\eps
  		-\dfrac{1}{\theta}\dfrac{(v-u_\eps) M_{u_\eps}}{ r_\eps}\left(\theta \partial_x\rho_\eps - \rho_\eps u_\eps\partial_x u_\eps\right).
\end{equation*}
For further purposes, observe that 
\begin{equation}\label{eq:intv2g}
	\begin{aligned}
 	\int v^2 g_\eps\ud v &=\int (v-u_\eps)^2 g_\eps \ud v + 2u_\eps \int v g_\eps\ud v \\
 	&= \dfrac{2 u_\eps}{r_\eps} \left(\rho_\eps u_\eps\partial_x u_\eps-\theta \partial_x\rho_\eps+ \right)
		-\theta \rho_\eps\partial_x u_\eps \,.
    	\end{aligned}
\end{equation}
 We are now going back to the hydrodynamic 
 system \eqref{eq:hyd1_1}-\eqref{eq:hyd1_2}, where we similarly get rid of terms of order higher than $\mathscr O(\eps)$.
 To this end, we introduce a convenient change of variables  by setting  
\begin{equation*}
	w_\eps := r_\eps u_\eps + \eps \int vg_\eps\ud v.
\end{equation*}
Moreover, we shall replace the quantities arising in the previous expression by their first order approximations:
\begin{equation*}
	\begin{aligned}
 	&\partial_x u_\eps		&\rightsquigarrow\quad&		- \frac{w_\eps}{r_\eps^2}\,\partial_x \rho_\eps+\frac{1}{r_\eps}\,\partial_ x w_\eps\,,\\
  	&u_\eps \partial_x u_\eps	&\rightsquigarrow\quad& 		\frac{w_\eps }{r_\eps^2}\partial_ x w_\eps- \frac{w^2_\eps}{r_\eps^3}\partial_x \rho_\eps\,,
	\end{aligned}
\end{equation*}
and
\begin{equation*}
	\begin{aligned} 
   	&\eps \int v g_\eps\ud v		&\rightsquigarrow\quad& 		I_{1,\eps}= \frac\eps{r_\eps}
		\left(\frac{ \rho_\eps w_\eps }{r_\eps^2}\partial_ x w_\eps
		- \frac{ \rho_\eps w^2_\eps}{r_\eps^3}\partial_x \rho_\eps-\theta\partial_x \rho_\eps\right),\\
 	&\eps \int v^2 g_\eps\ud v		&\rightsquigarrow\quad&		I_{2,\eps}=-\eps \theta\rho_\eps\left(\frac{1}{r_\eps}\partial_ x w_\eps
		-\frac{w_\eps}{r_\eps^2}\partial_x \rho_\eps\right)+\frac{2w_\eps}{r_\eps} I_{1,\eps},
	\end{aligned}
\end{equation*}
where the last two expressions should be compared to \eqref{eq:intvg} and \eqref{eq:intv2g}, respectively.
Therefore, based on these approximations, equality \eqref{eq:hyd1_1} leads to 
 \begin{equation}\label{eq:diff1}
 	\begin{aligned}
	 \partial_t\rho_\eps + \partial_x \left(\frac{\rho_\eps w_\eps}{r_\eps}\right)
	 &=\eps \partial_x\left\{\left(\frac{\rho_\eps}{r_\eps}-1\right)I_{1,\eps} \right\}\\
 	&=\eps \partial_x\left\{\left(\frac{\rho_\eps w_\eps^2}{r_\eps^5}+\frac{\theta}{r_\eps^2}\right)
		 \partial_ x\rho_\eps-\frac{\rho_\eps w_\eps }{r_\eps^4} \partial_ x w_\eps\right\}.
	 \end{aligned}
\end{equation}  
 Next, for relation \eqref{eq:hyd1_2}, approximating $u_\eps^2 $ by $\dfrac{w_\eps^2}{r_\eps^2} -  \dfrac{2\eps\,w_\eps^2}{r_\eps^2}I_{1,\eps}$, we get
 \begin{equation}\label{eq:diff2}
 	\begin{aligned}
	\partial_t w_\eps + \partial_x \left(\frac{w_\eps ^2}{r_\eps} +\theta \rho _\eps\right)
	&= -\eps \left(I_{2,\eps} - \frac{2w_\eps}{r_\eps}I_{1,\eps}\right)\\
	&=\eps \partial_x \left(-\frac{\theta \rho_\eps w_\eps} {r_\eps^2}\partial_x \rho_\eps
		+\frac{\theta \rho_\eps}{r_\eps} \partial_x w_\eps\right).
	\end{aligned}
\end{equation}  
Dropping  for shortness the dependence with respect to $\eps$, we end up with the second-order system 
in the  variable $\W=(\rho,w)$ which is
\begin{equation}\label{eq:sys_fin}
	\partial_t \W +\partial_x F(\W)=\eps\partial_x\bigl\{\mathbf{D}(\W)\partial_x  \W \bigr\}
\end{equation}
with the flux $F$ given in \eqref{eq:WandF_Bfp} and the diffusion matrix $\mathbf{D}$ defined as
\begin{equation}\label{eq:diffusionBfp}
	\mathbf{D}(\W):=\mathbf{D}_0(\W)+\theta\,\mathbf{D}_1(\W),
\end{equation}
where
\begin{equation*}
	\mathbf{D}_0(\W):=\frac{\rho w}{r^5}	\begin{pmatrix}	w & -r \\ 0 & 0 \end{pmatrix}
						=\frac{\rho u}{r^3}	\begin{pmatrix}	u & -1 \\ 0	& 0 	\end{pmatrix}
\end{equation*}
and
\begin{equation*}
	\mathbf{D}_1(\W):=\frac{1}{r^2}	\begin{pmatrix} 1 & 0 \\ -\rho w & \rho\,r \end{pmatrix}
						=\frac1r\begin{pmatrix}	1/r	& 0	\\ -\rho u			&  \rho	\end{pmatrix}
\end{equation*}

\begin{rmk}\rm 
Since system \eqref{eq:hyp_noncons} is not invariant under Galilean transformations, the same curse occurs for the extended model \eqref{eq:sys_fin}. 
Moreover, it can be easily checked that invariance with respect to space reversal also holds for such a higher order system.
\end{rmk}

Once more, recalling \cite[Corollary 2.2]{MaPe} and having already verified that $\ud^2 \eta\ud F$ is symmetric,
it is enough to show that, choosing $\mathbf{X}:=\ud^2 \eta$, the modified diffusion term $\mathbf{X}\mathbf{D}$ is positive definite.
Indeed, using the shorthand notation $\chi=1+\rho r$, we compute the composition
 \begin{equation*}
	 \mathbf{X}\mathbf{D}=\frac{1}{\rho\,r^4}\begin{pmatrix}
	\rho^2 u^4+\theta(1+\chi)\rho\,r u^2+\theta^2r^2 &-\rho u(\rho u^2+\theta\chi r)\\
 	-\rho u(\rho u^2+\theta\chi r) &\rho^2(u^2+\theta\,r^2)
 		\end{pmatrix}
\end{equation*}
which we observe to be symmetric too.
Moreover, the trace $\textrm{tr}(\mathbf{X}\mathbf{D})$ is clearly strictlypositive for $\rho>0$ and $\theta>0$. 
By the Binet Theorem for determinants, there holds
\begin{equation*}
	\det\bigl(\mathbf{X}\mathbf{D}\bigr)
		=\det\mathbf{X}\cdot\det\mathbf{D}\\
		=\frac{\theta}{\rho r}\left(\dfrac{\theta\,\rho^2 u^2}{r^3}+\dfrac{\theta^2\rho}{r^2}-\dfrac{\theta\,\rho^2\,u^2}{r^3}\right)
		=\dfrac{\theta^3}{r^3},
\end{equation*}
having used the explicit form of $\mathbf{D}$ given in \eqref{eq:diffusionBfp}.
Therefore, we infer that $\mathbf{X}\mathbf{D}$ is symmetric and positive-definite for $\theta>0$.
 
We summarize our findings in a concise statement.

\begin{prop}\label{prop:stab_bfp}
For any $\theta>0$, then system \eqref{eq:sys_fin} with the flux $F$ given in \eqref{eq:WandF_Bfp} and
the diffusion matrix $\mathbf{D}$ as in \eqref{eq:diffusionBfp} is strictly stable in the sense of  Definition~\ref{def:def2}.
\end{prop}
 
Next, having already verified the validity of Liu's entropy conditions,  we directly apply \cite[Corollary~2]{MaPe}
to establish the existence of weak viscous shocks, i.e. a solution to the two-dimensional ODE system
\begin{equation}\label{eq:profile}
	\eps\mathbf{D}(\mathrm W)\frac{\dd \mathrm W}{\dd\,y}=F(\mathrm W)-F(\W_\ast)- c(\mathrm W-\W_\ast),
\end{equation}
satisfying the asymptotic conditions
\begin{equation}\label{eq:asymptotics}
	\lim_{y\to -\infty} \mathrm W(y)=\W_\ast,\qquad \lim_{y\to +\infty} \mathrm W(y)=\W_\times
\end{equation}
with the propagation speed $c$ given by the Rankine--Hugoniot conditions \eqref{eq:RH0} and $\W_\times$
sufficiently close to $\W_\ast$.

We summarize our result in a synthetic statement whose proof follows from the results taken from \cite{MaPe}
together with the discussion relative to the validity of Liu's entropy condition \eqref{eq:liu}.

\begin{theorem}\label{th:th1}
Let the triple $(\W_\ast,\W_\times,c)$ be such that the Rankine--Hugoniot conditions \eqref{eq:RH0} is satisfied.
The strictly stable system \eqref{eq:efp} supports weak shock profiles -- i.e. there exists $\delta>0$ such that  if $|\W_\times- \W_\ast|\leq \delta$
there exists a function $y\mapsto \mathrm W^\eps(y)$ with ${\sup_{y\in\R}|\mathrm W^\eps(y)-\W_\ast|\leq \delta}$,
solution to \eqref{eq:profile} with asymptotics \eqref{eq:asymptotics}-- if and only Liu's criterion \eqref{eq:liu} is satisfied,
that is, $\bar{\rho}<\rho_\ast$ for the sign $+$ in the choice of $c$, and $\rho_\ast<\bar{\rho}$ for the sign $-$.
\end{theorem}

Using the appropriate unknowns (specifically, the entropy variables) is crucial to obtain the existence result stated in Theorem~\ref{th:th1}.
Different coordinates could support incorrect conclusions. 
Among others, a detailed discussion on stability properties of weak propagation fronts proved in Theorem~\ref{th:th1} can be found in \cite{ZumbHowa98}.

\subsection{A  few remarks on the  stability estimate}\label{sec:comm_par}

Let us go back to the discussion in subsection \ref{sec:sub_enest} to further illustrate some relevant implications in the case
of the viscous Burgers fluid-particle problem \eqref{eq:bfp}.
At first sight, even if tempting, requiring $\mathbf{D}$ in \eqref{eq:diffBfp}  to be  parabolic in the sense of Definition \ref{def:parab}
involves (unphysical) limitations on the temperature, as shown in the following claim.

\begin{prop}\label{lem:DsymPosDef}
Let $\mathbf{D}$ be defined in \eqref{eq:diffusionBfp} and set $\Lambda:=\sqrt{\rho\,r u^2}$.
The symmetric part  $\mathbf{D}_{\textrm{\rm sym}}$ of $\mathbf{D}$ is strictly positive definite if and only if
\begin{equation*}
	\theta\in\left\{\begin{aligned}
		&(\theta_1,+\infty) 	&\quad	&\textrm{if}\quad 0\leq \Lambda\leq 2,\\
		&(\theta_1,\theta_2)	&\quad	&\textrm{if}\quad \Lambda >2,\\
		\end{aligned}\right.
\end{equation*}
with $\theta_1=\theta_1(\U):=r^{-2}\Lambda/(\Lambda+2)$ and $\theta_2=\theta_2(\U):=r^{-2}\Lambda/(\Lambda-2)$.
\end{prop}

\begin{figure}[!htb]
\includegraphics[height=5cm]{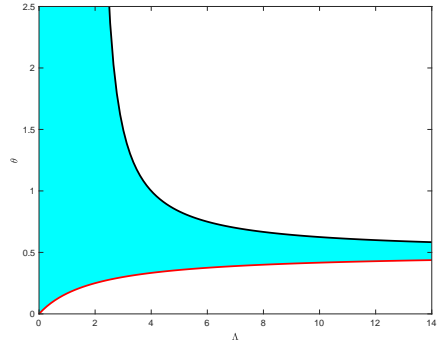}
\caption{\footnotesize  Admissible region in the $(\Lambda,\theta)-$plane where $\mathbf{D}_{\textrm{sym}}$
 is strictly positive definite for the choice $\rho=1$.}
\label{fig:Lemma34}
\end{figure}

This has to be compared to Proposition~\ref{prop:stab_bfp}, concluding that the notion of parabolicity 
provided in Definition \ref{def:parab} is not the appropriate notion to investigate the stability of viscous
perturbations of hyperbolic problems.
On the one hand, as explained in Section~\ref{sec:sub_enest} it is not enough to obtain stability 
estimates which are uniform with respect to $\epsilon$. On the other hand, 
it might involve irrelevant restrictions on the parameters of the problem.

\begin{proof}
It is readily seen that the trace of the matrix  $\mathbf{D}_{\textrm{sym}}$, that is
\begin{equation*}
 	\textrm{tr}\left( \mathbf{D}_{\textrm{sym}}\right)=\textrm{tr}\left( \mathbf{D}\right)
		=r^{-3}\left\{\rho u^2+\theta r(1+\rho\,r)\right\}\,,
\end{equation*}
is positive for any $\rho\geq 0$ and $\theta>0$.
The determinant of the symmetric part $\mathbf{D}_{\textrm{sym}}$ can be regarded
as a second-order polynomial with respect to $\theta$:
\begin{equation*}
	P(\theta)= \tfrac{1}{4}\rho r^{-3}Q(\theta)
	\qquad\textrm{where}\quad 
	Q(\theta):=(4-\Lambda^2)\theta^2+2r^{-2}\Lambda^2\theta-r^{-4}\Lambda^2\,.
\end{equation*}
Since the reduced discriminant of $Q$ is ${\Delta}/{4}={4\Lambda^2}/r^4$, we infer the factorization
\begin{equation*}
	Q(\theta)=\left\{(2-\Lambda)\theta+{\Lambda}/{r^{2}}\right\}
			\left\{(2+\Lambda)\theta-{\Lambda}/{r^{2}}\right\}.
\end{equation*}
In particular, the symmetric matrix $\mathbf{D}_{\textrm{sym}}$ is strictly definite positive if and only if $Q(\theta)>0$
providing the above restrictions on the parameter $\theta$.
\end{proof}

\section{Flowing regime for  the Euler fluid-particle system}
\label{sec:Euler} 

A more realistic model couples the evolution of the particles, with the Euler equation for the carrier fluid.
Namely, we consider 
\begin{equation}\label{eq:kinEu}
	\partial_t f_\eps+ v\partial_x f_\eps =\dfrac{1}{\eps} L_{u_\eps}(f_\eps)\,,
\end{equation}
coupled to 
\begin{equation}\label{eq:Eul}
	\left\{\begin{aligned}
	&\partial_t n_\eps+ \partial_x (n_\eps u _\eps) =0\,,\\
	&\partial_t (n_\eps u _\eps)+ \partial_x \left\{n_\eps u _\eps^ 2 + p(n_\eps)\right\}
		=-\dfrac{1}{\eps} \int vL_{u_\eps}(f_\eps)\ud v =\dfrac{1}{\eps} (J_\eps - \rho_\eps u_\eps)\,,
	\end{aligned}\right.
\end{equation}
still with the notation 
\begin{equation*}
	\rho_\eps=\int f_\eps \ud v
	\quad\textrm{and}\quad
	J_\eps=\int vf_\eps \ud v\,.
\end{equation*}
Here the unknown $n_\eps$ stands for the density of the carrier fluid, and $u_\eps$ for its velocity field.
The pressure function $p=p(n)$ obeys the standard principles of thermodynamics:
it is increasing and strictly convex, 
a typical example being the $\gamma$-law given in \eqref{eq:gammalaw}.

\subsection{Derivation and hyperbolicity}

Again, as $\eps$ goes to 0, we infer heuristically that
\begin{equation*}
	f_\eps(t,x)\simeq \rho_\eps M_{u_\eps(t,x)}(v)\,,
\end{equation*}
where $M_u$ is  the Maxwellian distribution introduced in \eqref{eq:maxwell}.
Hence, setting $r:=n+\rho$ and $w:=ru$, the limiting quantity $\W=(r,\rho,w)$ satisfies at leading order
the extended  nonlinear system \eqref{eq:ext}, which has the form \eqref{eq:hyp_cons} where the flux $F$ is given by
\begin{equation}\label{eq:eul_flux}
	F(\W):=\left(w,\rho w/r,w^2/r+p(n)+\theta \rho\right).
\end{equation}
We refer the reader to \cite{CG} for the introduction of this model; further numerical investigation can be found in \cite{CGL}.

Following again the standard approach, we verify that the extended system \eqref{eq:ext} is hyperbolic,
i.e. the Jacobian $\ud F=\ud F(\W)$, explicitly given by
\begin{equation*}
	\ud F=\begin{pmatrix} 0 & 0 & 1 \\ -\rho w/r^2 & w/r & \rho/r \\ -w^2/r^2 + p' & -p'+\theta & 2w/r \end{pmatrix}
		=\begin{pmatrix} 0 & 0 & 1 \\ -\rho u/r & u & \rho/r \\ -u^2+p' & -p'+\theta & 2u \end{pmatrix}
\end{equation*}
is such that
\begin{equation*}
	\det(\ud F - \lambda\mathbf{I})
	=-(\lambda-u)\left\{(\lambda-u)^2-(np'+\theta\rho)/r\right\}
\end{equation*}
so that the eigenvalues are real, being explicitly given by
\begin{equation}\label{eq:eigenvEuler}
	\lambda=\lambda_0=u\qquad\textrm{and}\qquad
	\lambda_\pm=u\pm\sqrt{(np'+\theta\rho)/r}\,.
\end{equation}

\begin{rmk}\label{rmk:galinv1}\rm
Set $(y,s)=(x-u_0 t,t)$, corresponding to $(\partial_x,\partial_t)=(\partial_y,\partial_s-u_0 \partial_y)$ and set $v:=u-u_0$.
The first two equations in \eqref{eq:ext} are invariant with respect to Galilean transformations. 
Indeed, there holds
\begin{equation*}
	\partial_t r +\partial_x(r u) = \partial_s r -u_0\partial_y r + \partial_y\bigl\{r(v+u_0)\bigr\}
		= \partial_s r +\partial_y(r v),
\end{equation*}
with an analogous computations for the unknown $\rho$.
Concerning the third equation, introducing the {\it total pressure} $P:= p + \theta \rho$, there holds
\begin{equation*}
	\begin{aligned}
	\partial_t (ru) 
	& +\partial_x(ru^2+P)
		= \partial_s \bigl\{r(v+u_0)\bigr\} -u_0\partial_y\bigl\{r(v+u_0)\bigr\} + \partial_y \bigl\{r(v+u_0)^2+P\bigr\}\\
	&\hskip1.1cm = \partial_s (rv)+u_0\partial_s r -u_0\partial_y(rv)-u_0^2 \partial_y r + \partial_y \bigl\{r(v^2+2u_0 v+u_0^2)+P\bigr\}\\
	&\hskip1.1cm = \partial_s (rv) + \partial_y (rv^2+P) -2u_0\partial_y(rv)-u_0^2 \partial_y r + 2u_0\partial_y(rv) + u_0^2 \partial_y r \\
	&\hskip1.1cm = \partial_s (rv) + \partial_y (rv^2+P),
	\end{aligned}
\end{equation*}
showing that the hyperbolic system \eqref{eq:ext} is invariant with respect to
Galilean transformations.
In addition, it can also be shown that the above system is invariant under space reversal,
the proof being very similar to the one for the reduced system \eqref{eq:red}.
\end{rmk}

In parallel with Proposition~\ref{prop:redfields}, we are now interested in a more precise classification
of the characteristic fields for the conservation law system \eqref{eq:ext}.

\begin{prop}\label{prop:Efp_fields}
Let assumption \eqref{eq:hyppressure} be satisfied.
Then, for any $\theta\geq 0$, system \eqref{eq:ext} is strictly hyperbolic with one linearly degenerate
field and two genuinely nonlinear fields whenever $n>0$ and $\rho, \theta\geq0$ or $n=0$ and $\rho, \theta>0$.
\end{prop}

\begin{proof}
To start with, let us compute $\nabla_{\W}\lambda$ for $\lambda\in\{\lambda_0,\lambda_\pm\}$.
Upon computations, we infer
\begin{equation*}
	\nabla_{\W}\lambda_0=\left(-\frac{w}{r^2},0,\frac{1}{r}\right)
	\quad\textrm{and}\quad  
	\nabla_{\W}\lambda_\pm=\left(-\frac{w}{r^2}\pm \frac{np'' r+\rho(p'-\theta)}{2dr^2},
		\mp\frac{p'+np''-\theta}{2dr},\frac{1}{r}\right)\,,
\end{equation*}
where $d:=\sqrt{(np'+\theta\rho)/r}$.
Relying on the Galilean invariance, we may reduce to the case $u=0$ (corresponding to $w=0$), hence
upon computations, we infer $\lambda_0=0$ and $\lambda_\pm=\pm d$ together with
\begin{equation*}
	\nabla_{\W}\lambda_0=\left(0,0,\frac{1}{r}\right)
	\quad\textrm{and}\quad
	\nabla_{\W}\lambda_\pm=\left(\pm \frac{np'' r+\rho(p'-\theta)}{2dr^2},
		\mp\frac{p'+np''-\theta}{2dr},\frac{1}{r}\right)\,.
\end{equation*}
Right eigenvectors relative to $\lambda_0$ are proportional to the vector $\mathbf{r}_0:=(p'-\theta,p',0)^\intercal$.
Since
\begin{equation*}
	\nabla_{\W}\lambda_0\cdot \mathbf{r}_0=0\cdot (p'-\theta)+0\cdot p'+\dfrac{1}{r}\cdot 0=0,
\end{equation*}
the field $\lambda_0$ is linearly degenerate.

Right eigenvectors relative to eigenvalues $\lambda_\pm$ are proportional to $\mathbf{r}_\pm:=(1,{\rho}/{r},\pm d)^\intercal$.
Therefore, there holds
\begin{equation*}
	\begin{aligned}
	\nabla_{\W}\lambda_\pm\cdot \mathbf{r}_\pm
		&=\pm \frac{np'' r+\rho(p'-\theta)}{2dr^2}\cdot 1 \mp\frac{p'+np''-\theta}{2dr}\cdot \frac{\rho}{r}\pm \frac{1}{r}\cdot d\\
		&=\pm\left\{\frac{np''\,r+\rho(p'-\theta)-(np''+p'-\theta)\rho}{2dr^2}+\frac{d}{r}\right\} 
			=\pm \frac{n^2 p''+2np'+2\theta\rho}{2dr^2}\neq 0\,,
	\end{aligned}
\end{equation*}
for any $n>0$ and $\rho,\theta\geq 0$ or $n=0$ and $\rho, \theta>0$.
In particular, the characteristic fields $\lambda_\pm$ are genuinely nonlinear in such a regime.
\end{proof}

\subsection{Shock solutions}
\label{Sh_sol}

To investigate discontinuous solutions, we again take advantage of relations \eqref{eq:jumpuseful}.
Having fixed a state $(\rho,n,u)\neq (\rho_\ast,n_\ast,u_\ast)$, 
the Rankine-Hugoniot conditions associated to system \eqref{eq:ext} read 
\begin{equation}\label{eq:rh_3d}
	c\llbracket \rho \rrbracket = \llbracket \rho u \rrbracket,\quad 
	c\llbracket n \rrbracket = \llbracket n u \rrbracket,\quad
	c \llbracket r u \rrbracket = \llbracket r u^2+ \theta \rho + p(n) \rrbracket.
\end{equation}

\begin{lemma}\label{lem:zerojump}
The following implications  hold true.
\begin{itemize}
\item[\bf i.] If one among the quantities $\llbracket \rho \rrbracket$, $\llbracket n \rrbracket$, $\llbracket r \rrbracket$ and $c-u_\ast$
is zero then $\llbracket u\rrbracket =0$.
\item[\bf ii.] If $\llbracket u\rrbracket =0$ and $\bigl(\llbracket \rho\rrbracket,\llbracket n\rrbracket,\llbracket r\rrbracket\bigr)\neq (0,0,0)$,
then $c=c_0:=u_\ast$.
\end{itemize}
\end{lemma}

\begin{proof}
{\bf i.}  There holds $c\llbracket \rho \rrbracket=\llbracket \rho u \rrbracket = \llbracket \rho \rrbracket  u_\ast + \rho \llbracket u \rrbracket$, 
hence
\begin{equation*}
	\llbracket \rho \rrbracket (c-u_\ast) = \rho \llbracket u \rrbracket,
\end{equation*}
and the conclusion follows.
A similar proof holds for $n$ and $r=\rho+n$, observing that, summing up the equations for $\rho$ and $n$, 
there holds $\partial_t r+\partial_x (ru)=0$ and $c\, \llbracket r\rrbracket = \llbracket ru\rrbracket$.

{\bf ii.} Since $c\llbracket \rho \rrbracket = \llbracket \rho u \rrbracket = \llbracket \rho \rrbracket u_\ast$,
the conclusion is trivial if $\llbracket \rho\rrbracket \neq 0$.
A similar argument can be invoked if $\llbracket n\rrbracket \neq 0$ and $\llbracket r\rrbracket \neq 0$ using
the analogous relation for $n$ and $r$.
\end{proof}

If $\llbracket \rho \rrbracket\neq 0$ and $\llbracket n \rrbracket\neq 0$, then, equations \eqref{eq:rh_3d} are equivalent to
\begin{equation}\label{eq:rh_3d_eq}
	c=\frac{\llbracket \rho u \rrbracket}{\llbracket \rho \rrbracket}
		=\frac{\llbracket n u \rrbracket}{\llbracket n \rrbracket}
		=\frac{\llbracket ru^2+\theta\rho+p(n) \rrbracket}{\llbracket ru \rrbracket}.
\end{equation}  
As proved in the following result, such shock solutions enjoy {\it Liu's entropy condition}
under appropriate standard assumptions on the pressure $p$.

\begin{prop}\label{prop:liucond}
If $\llbracket u \rrbracket\neq 0$ then the speed $c$, given in the equalities \eqref{eq:rh_3d_eq}, satisfies Liu's entropy condition.
\end{prop}
 
\begin{proof}
From the second equality in \eqref{eq:rh_3d_eq}, we infer $\llbracket n u \rrbracket \llbracket \rho \rrbracket = \llbracket n \rrbracket \llbracket \rho u \rrbracket$,
which, after a straightforward computation, gives $nu\rho_\ast+n_\ast u_\ast\rho = n\rho_\ast u_\ast + u_\ast \rho u$.
In turn, the latter reduces to $(n\rho_\ast- n_\ast \rho ) \llbracket u \rrbracket = 0$ so that $n\rho_\ast = n_\ast \rho$.
Therefore, we obtain
\begin{equation}\label{eq:rho_n}
	\rho={n\rho_\ast}/{n_\ast}
	\qquad\textrm{and}\qquad
	\llbracket\rho\rrbracket ={\llbracket n\rrbracket \rho_\ast}/{n_\ast}.
\end{equation}
Recalling the identity $r=n+\rho$, from \eqref{eq:rh_3d_eq} it also follows
\begin{equation*}
	\llbracket r u^2+\theta\rho+p \rrbracket\llbracket \rho \rrbracket
		= \llbracket \rho u \rrbracket \llbracket ru \rrbracket
\end{equation*}
with a similar relation holding for $n$ in place of $\rho$, so that, summing up,
\begin{equation}\label{eq:post_rh_3d}
	\llbracket r u^2+\theta\rho+p \rrbracket \llbracket r \rrbracket = \llbracket r u \rrbracket^2.
\end{equation}
The first term on the lefthand side of \eqref{eq:post_rh_3d} can be rewritten as
\begin{equation*}
	\begin{aligned}
	\llbracket ru^2+\theta\rho+p \rrbracket
	& = r \llbracket  u \rrbracket^2  + 2 r u_\ast \llbracket  u \rrbracket +  \llbracket r \rrbracket u_\ast^2  + \theta\llbracket \rho\rrbracket +\llbracket p \rrbracket.
	\end{aligned}
\end{equation*}
Similarly, there holds $\llbracket r u \rrbracket^2 = \bigl(r\llbracket u\rrbracket + \llbracket r\rrbracket u_\ast \bigr)^2$.
Hence, plugging into \eqref{eq:post_rh_3d}, we infer
\begin{equation*}
	\begin{aligned}
	 r \llbracket r \rrbracket \llbracket  u \rrbracket^2  + 2r\llbracket r \rrbracket u_\ast   \llbracket  u \rrbracket+ \llbracket r \rrbracket^2 u_\ast^2 
	 & + \llbracket r \rrbracket\left\{\theta\llbracket \rho\rrbracket +\llbracket p \rrbracket\right\}\\
	&= r^2\llbracket u\rrbracket^2 +2r\llbracket r\rrbracket  u_\ast \llbracket u\rrbracket +\llbracket r\rrbracket^2 u_\ast^2, 
	\end{aligned}
\end{equation*}
that is,
\begin{equation*}
	\llbracket  u \rrbracket^2 = \frac{\llbracket r \rrbracket}{r_\ast r}\,\bigl(\theta\llbracket \rho\rrbracket +\llbracket p \rrbracket\bigr).
\end{equation*}
Taking advantage of \eqref{eq:rho_n}, we infer
\begin{equation}
	\llbracket  u \rrbracket^2 = \frac{\llbracket n \rrbracket^2}
		{r_\ast n}\left(\frac{\theta\rho_\ast}{n_\ast} +\frac{\llbracket p \rrbracket}{\llbracket n\rrbracket}\right).
\end{equation}
The  right-hand side is non-negative provided $p$ is a non-decreasing function, which makes this relation 
 consistent. 
In particular, there holds
\begin{equation*}
	\left|\frac{\llbracket  u \rrbracket}{\llbracket n \rrbracket}\right|=\left\{ \frac{1}{r_\ast n}
		\left(\frac{\theta\rho_\ast}{n_\ast} +\frac{\llbracket p \rrbracket}{\llbracket n\rrbracket}\right)\right\}^{1/2}
\end{equation*}
and, as a consequence,
\begin{equation*}
	c=u_\ast+n\frac{\llbracket  u \rrbracket}{\llbracket  n \rrbracket}
		=u_\ast\pm\left\{ \frac{n}{r_\ast}\left(\frac{\theta\rho_\ast}{n_\ast} +\frac{\llbracket p \rrbracket}{\llbracket n\rrbracket}\right)\right\}^{1/2}
		=:c_\pm(n).
\end{equation*}
Differentiating $c_+$ with respect to $n$, we deduce
\begin{equation*}
	\begin{aligned}
	\partial_n c_+&=\frac12\left\{r_\ast n\left(\frac{\theta\rho_\ast}{n_\ast} +\frac{\llbracket p \rrbracket}{\llbracket n\rrbracket}\right)\right\}^{-1/2}
		\left\{\frac{\theta\rho_\ast}{n_\ast} +\frac{\llbracket p \rrbracket}{\llbracket n\rrbracket}+n\,\frac{\ud}{\ud n}\frac{\llbracket p \rrbracket}{\llbracket n\rrbracket}\right\}.
	\end{aligned}
\end{equation*}
 Since  $p$ is strictly convex, there holds
\begin{equation*}
	\frac{\ud}{\ud n}\frac{\llbracket p \rrbracket}{\llbracket n\rrbracket}
	=\frac{\ud}{\ud n}\left\{\frac{p(n)-p(n_\ast)}{n-n_\ast}\right\}
	=\frac{p(n_\ast)-p(n)-p'(n)(n_\ast-n)}{(n_\ast-n)^2}>0.
\end{equation*}
In particular,  Liu's condition is satisfied for $c_+$ since $p''>0$.

A similar computation can be used to prove the same property for $c_-$.
\end{proof}

\subsection{Entropy for the inviscid Euler fluid-particle system}
\label{entropy}

Similarly to the \eqref{eq:bfp} case, the kinetic-fluid formulation suggests the functional
\begin{equation*}
	\zeta(\W)=\frac{w^2}{2r}+\Pi(n)+\theta\rho\ln \rho
	\quad\textrm{with}\quad
	\Pi(n):=\int_0^n \int_0^s \frac{1}{\varsigma}\frac{\ud p}{\ud\varsigma}(\varsigma)\,\ud \varsigma \ud s
\end{equation*}
as an entropy  for system \eqref{eq:ext}, see \cite{CG}.
For later use, we stress the identity $\Pi''=p'/n$.

In the special case of isentropic flows with pressure $p$ given by the standard $\gamma$-law, 
i.e. $p(n)=Cn^\gamma$ with $\gamma>1$, there holds 
\begin{equation*} 
	\Pi(n)=C\gamma \int_0^n \int_0^s \varsigma^{\gamma-2}\ud \varsigma \ud s
		=\frac{C\gamma}{\gamma-1}\int_0^n s^{\gamma-1}\ud \varsigma \ud s
		=\frac{Cn^\gamma}{\gamma-1}.
\end{equation*}
The gradient $\nabla_{\W}\zeta$ of the entropy $\zeta$ is explicitly given by
\begin{equation*}
	\begin{aligned}
	\nabla_{\W}\zeta(\W)^\intercal
		&=\left(-{w^2}/{2r^2}+\Pi',-\Pi'+\theta(1+\ln \rho),{w}/{r}\right)\\
		&=\left(-{u^2}/{2}+\Pi',-\Pi'+\theta(1+\ln \rho),u\right)\,,
	\end{aligned}
\end{equation*}
while the hessian $\DD^2_{\W}\zeta$ is 
\begin{equation*} 
	\DD^2_{\W}\zeta(\W) =\begin{pmatrix}
	{w^2}/{r^3}+\Pi''		& -\Pi''			& -{w}/{r^2}	\\
	-\Pi''				& \Pi''+\theta/\rho 	&  0			\\
	-{w}/{r^2}			& 0				& {1}/{r}		\end{pmatrix}
	=\begin{pmatrix}
	u^2/r+p'/n			& -p'/n			& -u/r	\\
	-p'/n				& p'/n+\theta/\rho 	&  0		\\
	-u/r				& 0				& {1}/{r}	\end{pmatrix}.
\end{equation*}
As before, tedious computations confirm that $\mathbf{X}:=\DD^2_{\W}\zeta$ symmetrizes the Jacobian $\ud F$
of the flux $\ud F$ of the hyperbolic system of conservation laws \eqref{eq:ext}.

\subsection{Viscous corrections leading to \eqref{eq:efp}} 

Again, we derive the second-order corrections associated to \eqref{eq:ext} by using the Chapman-Enskog expansion.
Namely, the function
\begin{equation*}
	g_\eps:=\dfrac{1}{\eps}\left(f_\eps - \rho_\eps M_{u_\eps}\right)
\end{equation*}
satisfies
\begin{equation}\label{eq:eqL1}
	\begin{aligned}
 	L_{u_\eps}(g_\eps)
	&= \eps\{\partial_t+v\partial_x\}g_\eps+M_{u_\eps}\left\{\partial_t \rho_\eps+ \partial_x (\rho_\eps u_\eps ) \right\}
		+\left(\frac{|v-u_\eps|^2}{\theta}-1\right)\rho_\eps  M_{u_\eps}\partial_x  u_\eps\\
	&\hskip5cm+ \frac{v-u_\eps}{\theta} \rho_\eps M_{u_\eps}\bigl\{\rho_\eps(\partial_t u_\eps + u_\eps\partial_x u_\eps) + \theta\partial_x \rho_\eps\bigr\}.
	\end{aligned}
\end{equation}
 Integrating the 
 kinetic equation yields
\begin{equation}\label{eq:massconsEu}
	\partial_t \rho_\eps+ \partial_x (\rho_\eps u _\eps) + \eps\partial_x \int vg_\eps \ud v=0.
\end{equation}
 Hence the first two terms in the right-hand side of \eqref{eq:eqL1}  
contributes only to the $\mathscr O(\eps)$ correction.
Next, by using  system \eqref{eq:Eul}, we get
\begin{equation*}
	\begin{aligned}
	\rho_\eps(\partial_t u_\eps + u_\eps\partial_x u_\eps)
		&=\frac{\rho_\eps}{n_\eps}\left\{\partial_t (n_\eps u _\eps)+ \partial_x (n_\eps u _\eps^ 2)\right\}
		=\frac{\rho_\eps}{n_\eps}\left\{-\partial_ xp + \int (v-u_\eps)g_\eps\ud v\right\}.
	\end{aligned}
\end{equation*}
Therefore, we arrive at 
\begin{equation*}
	\begin{aligned}
 	L_{u_\eps}(g_\eps) &= \left(\frac{|v-u_\eps|^2}{\theta}-1\right)\rho_\eps  M_{u_\eps}\partial_x  u_\eps \\
	&\quad +\frac{v-u_\eps}{\theta} \rho_\eps M_{u_\eps}\left(- \frac{\rho_\eps}{n_\eps}\partial_ xp
		+ \frac{\rho_\eps}{n_\eps} \int vg_\eps\ud v + \theta\partial_x \rho_\eps\right)
	+\mathscr O (\eps).
	\end{aligned}
\end{equation*}
Again, let us set $r_\eps:=\rho_\eps+n_\eps$.
Next, we multiply by $v$ and integrate in order to obtain a simple relation for $\int vg_\eps\ud v$, deducing
\begin{equation*}
	\int vg_\eps\ud v= \frac{n_\eps}{r_\eps} \left(\frac{\rho_\eps}{n_\eps}\partial_ xp
		-\theta\partial_x \rho_\eps\right)+\mathscr O (\eps)
\end{equation*}
and, consequently,
\begin{equation}\label{eq:eqg}
	g_\eps=-\frac1{2\theta}\left(|v-u_\eps|^2-\theta\right)\rho_\eps  M_{u_\eps}\partial_x  u_\eps
		-\frac{v-u_\eps}{\theta\,r_\eps} \rho_\eps M_{u_\eps}\left\{\theta n_\eps\partial_x \rho_\eps
		-\rho_\eps\partial_ x p\right\}+\mathscr O (\eps).
\end{equation}
 As a matter of fact, we have 
 \begin{equation}\label{eq:intv2gEul}
 \int v^2 g_\eps\ud v
 = 
 -\theta \rho_\eps\partial_x u_\eps + 2u_\eps \int vg_\eps\ud v +\mathscr O (\eps)
 .\end{equation}
 Finally, we express the conservation of the total momentum
\begin{equation}\label{eq:total_mt}
	\partial_ t w_\eps +\partial_x \left\{r_\eps u_\eps^2 +p+\theta\rho_\eps + \eps\int v^2g_\eps \ud v\right\}=0\,,
\end{equation}
where 
\begin{equation*}
	w_\eps := r_\eps u_\eps + \eps \int vg_\eps\ud v\,.
\end{equation*}
 We are now going to write  the hydrodynamic 
 system, which arises by  getting rid of the terms of order higher than $\mathscr O(\eps)$.
Thus, in the previous expressions we make use of the following approximations:
\begin{equation*}
	\begin{aligned}
 	u_\eps 				&\quad\rightsquigarrow\quad \frac{w_\eps}{r_\eps}-\frac{\eps n_\eps}{r_\eps^2}
		\left(\frac{\rho_\eps}{n_\eps}\,\partial_ xp- \theta\partial_x \rho_\eps\right)\,,\\
 	\partial_x u_\eps		&\quad\rightsquigarrow\quad \partial_x\left(\frac{w_\eps}{r_\eps}\right)
		= \frac{1}{r_\eps}\partial_ x w_\eps - \frac{w_\eps}{r_\eps^2} \partial_x r_\eps\,,
	\end{aligned}
\end{equation*}
and
\begin{equation*}
	\begin{aligned}	
   	u^2_\eps				&\quad\rightsquigarrow\quad \frac{w^2_\eps }{r_\eps^2} -\frac{2\eps n_\eps w_\eps}{r_\eps^2}
		\left(\frac{\rho_\eps}{n_\eps}\partial_ xp- \theta\partial_x \rho_\eps\right)\,,\\
  	\int v^2 g_\eps\ud v		&\quad\rightsquigarrow\quad -\theta\rho_\eps\partial_x\left(\frac{w_\eps}{r_\eps}\right)
  		- \frac{2 n_\eps w_\eps}{r_\eps^2}\left(\frac{\rho_\eps}{n_\eps}\partial_ xp- \theta\partial_x \rho_\eps\right)\,.
	\end{aligned}
\end{equation*}
Based on these approximations, we obtain a second-order system for $\W_\eps=(r_\eps,\rho_\eps,w_\eps)$
in the form \eqref{eq:hyppar_cons}, that is
\begin{equation}\label{eq:Eusysdiff}
	\partial_t \W_\eps +\partial_x F(\W_\eps)
	=\eps\partial_x\bigl\{\mathbf D(\W_\eps)\partial_x  \W_\eps\bigr\},
\end{equation}
where the flux $F$ is given in \eqref{eq:eul_flux} and the diffusion matrix $\mathbf{D}$ is given by \eqref{eq:diff_Efp},
which can also be decomposed as
\begin{equation}\label{eq:diff_Efp_munu}
	\mathbf{D}(\W)=\mathbf{D}_0(\W)+\theta\, \mathbf{D}_1(\W)\,,
\end{equation}
where, setting $\nu:=n/r \in(0,1)$, there holds
\begin{equation}\label{eq:defD0D1}
	\mathbf{D}_0:=\nu(1-\nu)p'\begin{pmatrix} 0 & 0 & 0 \\ -1  & 1 & 0 \\ 0 & 0 & 0 \end{pmatrix}
	\quad\textrm{and}\quad
	\mathbf{D}_1:= \begin{pmatrix} 0 & 0 & 0 \\ 0  & \nu^2 & 0 \\ -(1-\nu)u & 0 & 1-\nu \end{pmatrix}\,.
\end{equation}
The eigenvalues $\{\beta_0,\beta_1,\beta_2\}$ of the (triangular) diffusion matrix $\mathbf{D}$
are the element of its principal diagonal, viz.
\begin{equation*}
	\beta_0:=0\,,\quad \beta_1:=\nu(1-\nu)p'+\theta\nu^2\
	\quad\textrm{and}\quad \beta_2:=\theta(1-\nu).
\end{equation*} 
In particular, they are non-negative and,  differently from system \eqref{eq:bfp}, do not depend explicitly on the velocity $u$.

We can check the invariance with respect to the Galilean change of coordinates of system  \eqref{eq:Eusysdiff}.
Reformulating with respect to the  variable $\U=(r,\rho,u)$, we end up with
\begin{equation}\label{eq:Eusysdiff_NC}
	\partial_t G(\U) +\partial_x H(\U)
	=\eps\partial_x\left\{\mathbf E(\U)\partial_x  \U\right\}
\end{equation}
with $G(\U)=(r,\rho,ru)$, $H(\U)=(ru,\rho u,ru^2+p+\theta\rho)$ and
\begin{equation*}
	\mathbf{E}( \U):=
		\nu			\begin{pmatrix}  0	& 0	& 0	\\	-1	& 1		& 0	\\	0	& 0	& 0	\end{pmatrix}
		+\theta(1-\nu)	\begin{pmatrix}	 0	& 0	& 0	\\	0	& 1+\nu	& 0	\\	0	& 0	& r	\end{pmatrix}\,.
\end{equation*}
Introducing the variables $(y,s)$ and $u$ as in Remark~\ref{rmk:galinv1}, where we proved that
the left-hand side is invariant with respect to Galilean transformations,  we can also show that the
whole system \eqref{eq:Eusysdiff_NC} preserve the same property, as a consequence of the
independence of  $\mathbf{E}$ with respect to the velocity variable $u$. 

Since one of the eigenvalue of $\mathbf{D}$ is zero, the induced dissipation is partial and some additional stability is required.
In the present setting, the Kawashima--Shizuta condition --stating that there is no right eigenvector of $\ud F$
in the kernel of $\mathbf{D}$-- holds (see \cite{Kaw, KS}). 
Indeed, focusing without loss of generality on the case $u=0$,  the eigenvectors are proportional to
${\mathbf r=\left(1,\rho/r,\lambda\right)^\intercal}$ where $\lambda$ is a non-zero eigenvalue of $\ud F$
or to ${\mathbf r=\left(1,1-\theta/p',0\right)^\intercal}$ when $\lambda=0$.
Computing $\mathbf{D}\mathbf{r}$ for $\lambda\neq 0$ gives $(\mathbf{D}\mathbf{r})_{3}=\theta\lambda(1-\nu)\neq 0$
for $\theta>0$ and $\nu<1$.
Similarly, for $\lambda=0$, there holds $(\mathbf{D}\mathbf{r})_{2}=-\theta^2\nu^2/p'\neq 0$ for $\theta>0$ and $\nu>0$.
Summarizing, \eqref{eq:efp} satisfies the Kawashima--Shizuta stability condition for strictly positive temperature $\theta$
and $\nu$ in the open interval $(0,1)$ corresponding to $\rho$ and $n$ strictly positive.

For later use, let us also explore in more details  the temperature-less regime ${\theta=}0$.
In the case $\lambda\neq 0$, the third component  $(\mathbf{D}\mathbf{r})_{3}$ is null.
Nevertheless, the second component $(\mathbf{D}\mathbf{r})_{2}$ is equal to $-\nu^2(1-\nu)p'$
which is strictly negative if $\nu\in(0,1)$.
Hence the Kawashima-Shizuta condition holds for $\lambda\neq 0$.
Differently, for $\lambda=0$, there holds ${\mathbf{D}\mathbf{r}_0=(0,-\nu(1-\nu)p'+\nu(1-\nu)p',0)^\intercal=\mathbf{0}}$
and the condition is not satisfied.

Going further, we aim to show that the matrix $\ud^2_{\W}\zeta\,\mathbf{D}$ is symmetric. 
With this target, we rewrite the hessian $\ud^2_{\W}\zeta$ of the entropy  $\zeta$ (again with $u=0$,
thanks to the Galilean invariance) in terms of the scalar quantity $\nu=n/r$,
obtaining $\ud^2_{\W}\zeta=\mathbf{X}_0+\theta\mathbf{X}_1$ where
\begin{equation*}
	\mathbf{X}_0:=\frac{1}{n}\begin{pmatrix} p' & -p' & 0 \\ -p' & p' & 0 \\ 0 & 0 & \nu(1-\nu) \end{pmatrix}
	\qquad\textrm{and}\qquad
	\mathbf{X}_1:=\frac{\theta}{\rho}\begin{pmatrix} 0 & 0 & 0 \\ 0 & 1 & 0 \\ 0 & 0 & 0 \end{pmatrix}.
\end{equation*}
Then, we compute the matrix product
\begin{equation*}
	\mathbf{X}\mathbf{D}=(\mathbf{X}_0+\theta\mathbf{X}_1)(\mathbf{D}_0+\theta\mathbf{D}_1)
		=\mathbf{X}_0\mathbf{D}_0+\theta(\mathbf{X}_0\mathbf{D}_1+\mathbf{X}_1\mathbf{D}_0)+\theta^2\mathbf{X}_1\mathbf{D}_1.
\end{equation*}
Tedious computations bring the following final formulas
\begin{equation*}
	\mathbf{X}_0\mathbf{D}_0=\frac{\nu(1-\nu)(p')^2}{n}
		\begin{pmatrix} +1 & - 1 & 0 \\ -1 & +1 & 0 \\ 0 & 0 & 0 \end{pmatrix}
	\quad\textrm{and}\quad
	\mathbf{X}_1\mathbf{D}_1=\frac{\nu^2}{\rho}
		\begin{pmatrix} 0 & 0 & 0 \\ 0 & 1 & 0 \\ 0 & 0 & 0 \end{pmatrix}\,,
\end{equation*}
together with
\begin{equation*}
	\mathbf{X}_1\mathbf{D}_0+\mathbf{X}_0\mathbf{D}_1=\frac{1}{r}
		\begin{pmatrix} 0 & -\nu p' & 0 \\ -\nu p' & 2\nu p' & 0 \\ 0 & 0 & 1-\nu \end{pmatrix}\,,
\end{equation*}
showing the symmetry of the matrix $\mathbf{D}$.

\subsection{The temperature-less case}

As stated in the Introduction, the case where the Brownian velocity fluctuations are neglected  is  relevant in many applications.
Therefore, let us briefly discuss how the discussion adapts to handle the case $\theta=0$:
we  consider  system \eqref{eq:kinEu}-\eqref{eq:Eul}
where the Fokker-Planck operator in the right-hand side of \eqref{eq:kinEu} is replaced by $\partial_v\left\{(v-u_\eps) f_\eps\right\}$.
This does not modifiy the coupling term in \eqref{eq:Eul} which is still given by $J_\eps -\rho_\eps u_\eps$.
The ``equilibrium state'' that makes the stiff terms vanish is now a Dirac mass with respect to the velocity variable
\begin{equation*}
	f_\eps(t,x,v) \simeq \rho_\eps(t,x)\delta_{v=u_\eps(t,x)}.
\end{equation*}
This modifies the limit equation:
since $\int v^2 f_\eps\ud v \simeq \rho_\eps u_\eps^2$, there is no pressure term induced by the kinetic part of the equation and the limit 
equation becomes
\begin{equation}\label{eq:Eul_lim_th0}
	\left\{\begin{aligned}
	&\partial_t n+ \partial_x (n u) =0,\\
	&\partial_t \rho+ \partial_x (\rho u) =0, \\
	&\partial_t (r u)+ \partial_x \left\{ru ^ 2 + p(n)\right\} =0,
	\end{aligned}\right.
\end{equation}
instead of \eqref{eq:ext}.
Therefore, we can simply use the formula for the flux $F$ and the Jacobian matrix $\mathrm d F$ by setting $\theta=0$.
In particular, the  eigenvalues of $\mathrm dF$ become
\begin{equation}\label{eq:eigenvEuler_notemp}
	\lambda=\lambda_0=u,\quad \lambda_\pm=u\pm\sqrt{np'/r}.
\end{equation}
Accordingly we can set $\theta=0$ in the expressions of subsection~\ref{Sh_sol}.

We shall see that the conclusion is essentially the same for the viscous correction, but the computation
should be performed with some  caution.
The rationale consists in using the fact that $M_u$, defined in \eqref{eq:maxwell}, converges to a Dirac mass
$\delta_{v=u}$ as $\theta\to 0^+$ in the sense of distributions.
Accordingly, we also have
\begin{equation*}
	\begin{aligned}
	\lim_{\theta\to 0^+}\partial_v M_u &= - \lim_{\theta\to 0^+} \dfrac{1}{\theta}(v-u)\,M_u=\delta'_{v=u}\,,\\
	\lim_{\theta\to 0^+}\partial^2_{vv} M_u &= \lim_{\theta\to 0^+} \dfrac{1}{\theta^{2}} \left(|v-u|^2-\theta\right)M_u=\delta''_{v=u}\,,
	\end{aligned}
\end{equation*}
both being weak limits.
Thus, as $\theta\to 0^+$ in the right-hand side of \eqref{eq:eqL1} and in the remainder in \eqref{eq:eqg},
we infer that $g_\eps:=-\dfrac{\rho_\eps}{r_\eps}\partial_x p(n_\eps) \delta'_{v=u_\eps}$ is such that
\begin{equation*}
	\partial_v \left\{(v-u_\eps)g_\eps\right\}= \dfrac{\rho_\eps}{r_\eps} \partial_x p(n_\eps) \delta'_{v=u_\eps}\,,
	\qquad \int g_\eps\ud v=0,\qquad \int vg_\eps\ud v=\dfrac{\rho_\eps}{r_\eps} \partial_x p(n_\eps)\,.
\end{equation*}
Furthermore, the second order moment becomes
\begin{equation*}
	\int v^2g_\eps\ud v=\dfrac{2\rho_\eps u_\eps}{r_\eps} \partial_x p(n_\eps)\,.
\end{equation*}
By using the above formula, we obtain the closed equation \eqref{eq:Eusysdiff} with the  diffusion matrix \eqref{eq:diff_Efp_munu}
where we simply set $\theta=0$, i.e. $\mathbf{D}=\mathbf{D}_0$.

Let us stress that when $\theta=0$ the entropy $\zeta$ is convex but not strictly convex,
since we can easily check that $\mathbf X=\mathbf{X}_0$ is a singular matrix.
In particular, this precludes the possibility of applying the symmetrization method presented in Proposition~\ref{prop:ent_form}.

\subsection{Small-amplitude shock profiles analysis}

As in the previous computations, we may consider, without loss of generality,  a
co-moving frame  such that $u_\ast=0$.
To apply the result \cite[Theorem~4.1]{Pego84}, we focus on a genuinely nonlinear field $\lambda$ for system \eqref{eq:efp},
hence excluding the field $\lambda_0$ (see Proposition~\ref{prop:Efp_fields}).
For definiteness, let us concentrate on the case $\lambda=\lambda_+$, the case $\lambda=\lambda_-$ being similar.
For later convenience, let us recall the identity
\begin{equation}\label{eq:lambdapiusquared}
	\lambda_+^2=\nu p'+\theta(1-\nu)\quad\textrm{with}\quad \nu=n/r\in(0,1)\,.
\end{equation}
We are going to use  the following result, stated and proved in \cite{Pego84}, 
reported here for reader's convenience in a variation fitting the present context
(see \cite{Frei} for an alternative formulation).

\begin{theorem}[Theorem~4.1, \cite{Pego84}]\label{th:pego84}
Let $\boldsymbol{\ell}_+$ and $\mathbf r_+$ denote left and right eigenvectors of the matrix $\mathbf{A}$ relative to the eigenvalue $\lambda_+$, respectively.
In addition, let us assume
\begin{itemize}
\item[\bf i.] $\mathbf D(\W)$ has constant rank  in a neighborhood of  $\W_\ast$;
\item[\bf ii.] there holds $\boldsymbol{\ell}_+\mathbf D \mathbf r_+ (\W_\ast)\neq 0$;
\item[\bf iii.] the operator $\mathbf{B}(\xi):=i\xi(\mathbf{A}-\lambda_+\mathbf{I}) -\mathbf D$ is one-to-one on $\mathbb CZ$ for all $\xi\in\R$,
i.e. $\mathrm{Ker}\,\mathbf{B}(\xi)\bigr|_{\mathbb CZ}=\{0\}$, where 
\begin{equation}\label{eq:defZ}
	Z:=\big\{\mathbf{v}\in\R^3\,:\,(\mathbf{A}-\lambda_+\mathbf{I})\mathbf{v} \in \mathrm{Ran}\,\mathbf D\big\}\,;
\end{equation}
\end{itemize}
Then, the following are equivalent
\begin{itemize}
\item[\bf I.] there holds $\boldsymbol{\ell}_+\mathbf D \mathbf r_+(\W_\ast)>0$;
\item[\bf II.] there exists $\delta>0$ so that if $\W_\ast$ and $\W_\times$ are such that $|\W_\ast-\W_\times|<\delta$ 
and the Rankine--Hugoniot condition holds for some speed $c$, there exists a shock profile connecting $\W_\ast$ to $\W_\times$
if and only if Liu's entropy criterion \eqref{eq:gen_liu} is satisfied.
\end{itemize}
\end{theorem}

Verification of the above assumptions leads to the proof of existence
of shock profiles in the small amplitude regime.

\begin{theorem}\label{th:th2}
Let $\theta\geq 0$ and let $\W_\ast$ and $\W_\times$ are such that the Rankine--Hugoniot condition is satisfied for some speed $c$.
Then there exists $\delta>0$ so that there exists a shock profile solution to \eqref{eq:Eusysdiff}
connecting $\W_\ast$ to $\W_\times$ with $|\W_\ast-\W_\times|<\delta$. 
\end{theorem}

\begin{proof} 
The result is proved if  the assumption of Theorem \ref{th:pego84} are satisfied.
Without loss of generality, we consider the case $u=0$ by using once more the invariance with respect to Galilean transformations.

{\it Case $\theta=0$.}
For zero temperature, the matrix $\mathbf{D}$ reduces to $\mathbf{D}_0$ defined in \eqref{eq:defD0D1}.
Also, a triple of right/left eigenvectors of $\mathbf{A}$ relative to the eigenvalue $\lambda_k$ is given by $\mathbf{r}_{k}=(1,1-\nu,\lambda_{k})$
and $\boldsymbol{\ell}_{k}=(p',-p'+\theta,\lambda_{k})$ where $k\in\{0,\pm\}$.
Condition {\bf i.} in Theorem \ref{th:pego84} is clearly satisfied since  $\mathrm{Ran}\,\mathbf{D}(\W)$ coincides with
$\textrm{Span}\{\mathbf{e}_2\}$ for any $\W$ where $\{\mathbf{e}_1,\mathbf{e}_2,\mathbf{e}_3\}$ is the canonical basis of $\mathbb R^3$.
As a consequence, $\mathrm{Ran}\,\mathbf{D}(\W)$ has rank one.

Next, we state that  $Z$ coincides with $\mathrm{Span}\{\mathbf r_+\}$.
Indeed, let us consider the vector $\mathbf{v}=(x,y,z)\in\R^3$ such that $(\mathbf{A}-\lambda_+\mathbf{I})\mathbf{v}\in\mathrm{Ran}\,\mathbf{D}$.
Then there holds
\begin{equation*}
	(\mathbf{A}-\lambda_+\mathbf{I})\mathbf{v}
	=\begin{pmatrix} -\lambda_+ & 0 & 1 \\ 0 & -\lambda_+ & \rho/r \\ p' & -p' & -\lambda_+ \end{pmatrix}\begin{pmatrix} x \\ y \\ z \end{pmatrix}
	=\begin{pmatrix} -d x + z \\  -d y + \rho z/r \\ p' x -p' y -d z \end{pmatrix}
	=\alpha \mathbf{e}_2\,,
\end{equation*}
for some $\alpha\in\mathbb{R}$.
Plugging the relation $z=d x$, into the third component, we deduce the identity $y=\rho x/r$.
Finally, we insert both equations for $z$ and $y$, into the second component getting
\begin{equation*}
	-d y + \dfrac{1}{r}\rho z=-\dfrac{1}{r} d \rho x+\dfrac{1}{r} d \rho x = \alpha 
\end{equation*}
which implies $\alpha=0$.
In particular, the set $Z$ coincides with the one-dimensional eigenspace of the eigenvalue $\lambda_+$,
that is, $Z=\textrm{Ker}(\mathbf{A}-\lambda_+\mathbf{I})$.
Thus, we are required to analyze the kernel of the operator $\mathbf{B}(\xi)$ restricted to $Z$,
that is, we look for vectors $\mathbf{v}=\alpha\mathbf{r}_+$ for some $\alpha\in\mathbb{C}$ such that
$\mathbf{B}(\xi)\mathbf{v}=-\alpha\mathbf{D}\mathbf{r}_+=\mathbf{0}$.
Since the Kawashima--Shizuta condition is satisfied also for $\theta=0$, $\mathbf{D}\mathbf{r}_+\neq \mathbf{0}$ and, therefore, $\alpha=0$.
As a consequence, hypothesis {\bf iii.} is satisfied.

Finally, let us show that conditions {\bf iii.}/{\bf I.} are also verified.
Indeed, there holds
\begin{equation*}
	\boldsymbol{\ell}_+\mathbf D \mathbf r_+=\nu(1-\nu)p'\begin{pmatrix} p' & -p' & \lambda_+\end{pmatrix}
	\begin{pmatrix} 0 & 0 & 0 \\ -1  & 1 & 0 \\ 0 & 0 & 0 \end{pmatrix} \begin{pmatrix} 1 \\ 1-\nu \\ \lambda_+ \end{pmatrix}
	=\nu^2(1-\nu)(p')^2>0\,.
\end{equation*}

{\it Case $\theta>0$.}
For strictly positive temperatures, it is readily verified that $\mathrm{Ran}\,{\mathbf D}(\W)=\textrm{Span}\left\{\mathbf{e}_2, \mathbf{e}_3\right\}$
for any $\W$.
Hence, hypothesis {\bf i.} holds.
 
A vector $\mathbf{v}=(x,y,z)$ lies in $Z$ if and only if $z=\lambda_+ x$.
Therefore the action of the linear operator $\mathbf{B}(\xi)$ is described by
\begin{equation*}
	\begin{aligned}  
	\mathbf{B}(\xi)\mathbf{v}
	&=i\xi\begin{pmatrix} 0 \\ (1-\nu)\lambda_+ x-\lambda_+ y \\
          (p'-\lambda_+^2)x +(-p'+\theta)y \end{pmatrix} 
	+\begin{pmatrix} 0 \\ -\nu(1-\nu)p' x+\nu\{(1-\nu)p'+\theta\nu\}y\\
          \theta\lambda_+ (1-\nu)x \end{pmatrix} 
	\end{aligned}   
\end{equation*}
which can be rewritten as a reduced two dimensional system with coefficient matrix
\begin{equation*}
	\begin{aligned}
	\mathbf{M}&:=\begin{pmatrix}	i\xi(1-\nu)\lambda_+-\nu(1-\nu)p'			&	-i\xi\lambda_+ + \nu(1-\nu)p' +\theta\nu^2	\\ 
							i\xi(p'-\lambda_+^2)+\theta\lambda_+(1-\nu)	&	-i\xi(p'-\theta)					\end{pmatrix}\\
			&\;=\begin{pmatrix}	i\xi(1-\nu)\lambda_+-\nu(1-\nu)p'			&	-i\xi\lambda_+ + \nu(1-\nu)p' +\theta\nu^2	\\ 
							i\xi(1-\nu)(p'-\theta)+\theta\lambda_+(1-\nu)	&	-i\xi(p'-\theta)					\end{pmatrix}
	\end{aligned}
\end{equation*}
The real part of the determinant of $\mathbf{M}$ is
\begin{equation*}
	 \textrm{Re}(\det\mathbf{M})=-\theta\lambda_+\{\nu(1-\nu)p'+\theta\nu^2\}(1-\nu)
		<-\theta\lambda_+\nu(1-\nu)^2p'
\end{equation*}
which is strictly negative for any $\theta>0$ and $\nu\in(0,1)$.
Hence, the linear transformation $\mathbf{M}$ is a one-to-one correspondence, exhibiting the validity of {\bf iii.}

Finally, we compute explicitly the value of $\boldsymbol{\ell}_+\mathbf D \mathbf r_+>0$.
Since
\begin{equation*}
	\begin{aligned}
	\mathbf{D}\mathbf{r_+}&
	=\begin{pmatrix} 0 & 0 & 0 \\ -\nu(1-\nu)p'  & \nu(1-\nu)p'+\theta\nu^2 & 0 \\ 0 & 0 & \theta(1-\nu) \end{pmatrix}
		\begin{pmatrix} 1 \\ 1-\nu \\ \lambda_+ \end{pmatrix}
	=(1-\nu)\begin{pmatrix} 0 \\ -\nu^2(p'-\theta) \\ \theta\lambda_+ \end{pmatrix}\,,
	\end{aligned}
\end{equation*}
there holds
\begin{equation*}
	\begin{aligned}
	\boldsymbol{\ell}_+\mathbf{D}\mathbf{r}_+
		=(1-\nu)\nu^2(p'-\theta)^2+\theta^2\lambda_+\geq \theta^2\lambda_+>0.
	\end{aligned}
\end{equation*}
Thus, since Liu's entropy condition is satisfied (see Proposition~\ref{prop:liucond}),
we deduce the existence of small amplitude shock profiles as a consequence of Theorem~4.1 in \cite{Pego84}.
\end{proof}

Furthermore, conditions described in \cite{HumpZumb02} and \cite{MascZumb04a} are satisfied, so that the small
amplitude shock profiles are also asymptotically stable in some appropriate Sobolev space.

\section{Large amplitude profiles for viscous Euler fluid-particle system}
\label{num}

In this final Section, we continue the analysis relative to the existence of shock profiles for \eqref{eq:efp} in the large amplitude regime.
Such a choice is dictated by the fact that the model has the additional feature of being invariant with respect to Galilean transformations.
As a consequence, we can assume, without loss of generality, that the chosen reference frame is comoving
with the wave, i.e. the speed $c$ is equal to zero.
Hence, after the straightforward rescaling $x\mapsto \y:=x/\eps$, we search for a solution $\mathrm W=(r,\rho,w)$ of  
\begin{equation}\label{eq:profileEuler}
	\mathbf D(\mathrm W)\frac{\ud \mathrm W}{\ud \y} = F(\mathrm W)-F(\W_\ast),
\end{equation}
where the flux $F$ has been introduced in \eqref{eq:eul_flux} and the diffusion matrix $\mathbf{D}=\mathbf{D}_0+\theta \mathbf{D}_1$
with $\mathbf{D}_0$ and $\mathbf{D}_1$ defined in \eqref{eq:defD0D1}.
Moreover, we assume that the solution $\mathrm W$ is subjected to far-end states, denoted by $\W_\ast$
and ${\W}_\times$, which are related by the Rankine--Hugoniot conditions \eqref{eq:rh_3d}. 
Whether the far-end state of the asymptotic values $\W_\ast$ and ${\W}_\times$ is reached at $-\infty$
or at $+\infty$ will be made precise further on.

Since the first row of $\mathbf D$ vanishes, the first equation in \eqref{eq:profileEuler} imposes that $w$ is constant:
\begin{equation*}
	w=w_\ast:=r_\ast u_\ast\,.
\end{equation*}
We are thus led to a $2\times 2$ differential system for the pair $(r,\rho)$ given by
\begin{equation*}
	\left\{\begin{aligned}
	& -\frac{np'(n)\rho}{r^2}\frac{\ud r}{\ud\y}+\left\{\frac{np'(n)\rho}{r^2}+\frac{\theta n^2}{r^2}\right\}\frac{d\rho}{\ud\y} = \left(\frac{\rho}{r}-\frac{\rho_\ast}{r_\ast}\right)w_\ast\,,\\
	& -\theta\,\frac{\rho w_\ast}{r^2}\frac{\ud r}{\ud\y} = \frac{w_\ast^2}{r}+p(n)+\theta \rho-\frac{w^2_\ast}{r_\ast}-p(n_\ast)-\theta \rho_\ast\,,
	\end{aligned}\right.
\end{equation*}
which, on its turn, is equivalent to a system for the  pair $(r,n)$ that is
\begin{equation}\label{eq:ode_profile}
	\left\{\begin{aligned}
	&-\frac{\theta n^2}{r^2}\frac{\ud r}{\ud\y} + \left(\frac{\theta n^2}{r^2}+\frac{\rho\,np'(n)}{r^2}\right)\frac{\ud n}{\ud\y} = w_\ast\left( \frac{n}{r}-\frac{n_\ast}{r_\ast} \right),\\
	&-\frac{\theta \rho w_\ast}{r^2}\frac{\ud r}{\ud\y} = \frac{w_\ast^2}{r}+p(n)+\theta \rho-\frac{w^2_\ast}{r_\ast}-p(n_\ast)-\theta \rho_\ast
	\end{aligned}\right.
\end{equation}
Any solution to the dynamical system \eqref{eq:ode_profile} asymptotically converging to $\W_\ast$ and $\W_\times$
corresponds to a (smooth) shock profile for \eqref{eq:Eusysdiff}.

\subsection{Analysis of the temperature-less case}

When $\theta=0$ and $u_\ast\neq 0$, system \eqref{eq:ode_profile} degenerates to the scalar differential equation for $n$
\begin{equation}\label{eq:odeth0}
	\frac{(r-n) p'(n)}{r^2}\ \frac{\ud n}{\ud\y}  = u_\ast\left( \frac{r_\ast}{r}-\frac{n_\ast}{n}\right)\,,
\end{equation}
coupled with the identity 
\begin{equation}\label{eq:constraint}
	\frac{r}{r_\ast}=\frac{r_\ast u_\ast^2}{r_\ast u_\ast^2 + p_\ast-p(n)}\,. 
\end{equation}
where $p_\ast:=p(n_\ast)$.
We bear in mind that the function $r=n+\rho$ has the meaning of a {\it hybrid density}, being the sum
of the densities of the carrier and the disperse phases, denoted by $n$ and
$\rho$, respectively. The system degenerating to a single equation, we
  replaced $\rho$ by $r-n$ in \eqref{eq:odeth0}.
Accordingly, $r$ is required to satisfy the admissibility constraint $r>n$ for any ${n\in(0,\infty)}$, since $\rho=r-n>0$. 
 Under this constraint, one sees at once that the equilibrium states of \eqref{eq:odeth0} satisfy $r_\ast/r=n_\ast/n$. 

To make our computations on system
\eqref{eq:odeth0}--\eqref{eq:constraint}  easier to follow, we will introduce
rescaled variables. However, we will formulate our main theorem in the natural variables. Let
\begin{equation}\label{eq:variousdef:1}
        \n:=\dfrac{n}{n_\ast}\qquad\textrm{and}\qquad \r:=\dfrac{r}{r_\ast}\,,
\end{equation}
together with the auxiliary parameters
\begin{equation}\label{eq:variousdef}
	\tau:=\frac{n_\ast}{r_\ast}\in(0,1)\,,\qquad
	\kappa:=\frac{r_\ast u_\ast^2}{p_\ast}\in(0,\infty)\,,\qquad
	\kappa_\ast:=\frac{n_\ast p'_\ast}{p_\ast}\in(0,\infty)\,,
\end{equation}
where $p'_\ast=p'(n_\ast)$.
The parameter $\tau$ describes the ratio between the density of the disperse phase and the corresponding total density. 
 In particular, in term of the rescaled variables, the discussion about the sign of $\rho$ will then concern the one of $\r-\tau\,\n$.
The dimensionless number $\kappa$ is reminiscent of the Eckert number in fluid mechanics and
it compares the kinetic energy of the mixture to  the pressure of the carrier phase.
The value $\kappa_\ast$ is a given threshold separating different behaviors for the solution
of problem \eqref{eq:odeth0}--\eqref{eq:constraint}. 
Note that, once $n_\ast$ is fixed, $\kappa_\ast$ is completely determined.
Moreover, if $\tau$ is fixed, $r_\ast$ is also given.
Finally, if additionally $\kappa$ is fixed, the absolute value of $u_\ast$ is determined by the formula
\begin{equation}\label{eq:signuast}
	|u_\ast|=\sqrt{p_\ast\tau\kappa/n_\ast}\,.
\end{equation}
Finally, let us introduce the rescaled pressure
\begin{equation}
	\p(\n):=\frac{p(n_\ast \n)}{p_\ast}\,.
\end{equation}
Note that the function $\p$ shares the same monotonicity and convexity of $p$ and that
\begin{equation}\label{eq:tworel4p}
	\p(0)=0\,,\qquad
	\p(1)=1\,,\qquad
	\p'(1)=\kappa_\ast\,.
\end{equation}
Taking advantage of the previous definitions, the differential equation \eqref{eq:odeth0}
with constraint \eqref{eq:constraint} rewrites as
\begin{equation}\label{eq:rescaled}
	\left\{\begin{aligned}
	&\frac{u_*}{\kappa}\ \frac{\left(\r-\tau\n\right) \p'(\n)}{\r^2}\frac{\ud\n}{\ud \y}=\T(\n,\r):=\frac{1}{\r}-\frac{1}{\n}\,,\\
	&\r=\r_\kappa(\n):=\dfrac{\kappa}{1+\kappa-\p(\n)}\,.
	\end{aligned}\right.
\end{equation}
where  the function $\n\mapsto \r_\kappa(\n)$ is defined for $\n\in\left(0,\bar{\n}(\kappa)\right)$
with $\bar{\n}(\kappa):=\p^{-1}(1+\kappa)$.

\begin{lemma}\label{lem:ncross}
For any $\kappa>0$ with $\kappa\neq \kappa_\ast$ there exists a unique $\n(\kappa)\neq 1$ solution to $\g(\n):=\T(\n,\r_\kappa(\n))=0$.
Moreover, the function $\kappa\mapsto \n(\kappa)$ is one-to-one from $(0,\infty)\setminus\{\kappa_\ast\}$ to $(0,\infty)\setminus\{1\}$
with $\n(\kappa) < 1$ if and only if $\kappa<\kappa_\ast$. 
\end{lemma}

\begin{proof} 
For $\kappa\neq \kappa_\ast$, the function $\g$ is such that 
\begin{equation*}
	\lim_{\n\to 0^+} \g(\n)=-\infty,\qquad \g(1)=0 ,\qquad 
	\g'(1)=1-\frac{\kappa_\ast}{\kappa}\neq 0,\qquad
	\g(\bar{\n})=-\frac{1}{\bar{\n}}<0.
\end{equation*}
Moreover, the derivative
\begin{equation}\label{eq:gprime}
	\g'(\n)=\frac{1}{\n^2}-\frac{\p'(\n)}{\kappa}
\end{equation}
is decreasing in $\n$, hence the function $\g$ is concave.
(Graphs of the function $\g$ for several values of $\kappa$ are depicted in Fig.~\ref{fig:g_plot}, in the case of the pressure  
law \eqref{eq:gammalaw} 
with $\gamma=2$.)
In particular, for $\kappa<\kappa_\ast$, respectively $\kappa>\kappa_\ast$, there exists a unique
value $\n\in(0,1)$, respectively $\n\in(1,\bar{\n})$, such that $\g(\n)=0$.

Conversely, given $\n\in(0,+\infty)\setminus\{1\}$, let  $\kappa(\n)$ be such that $\g(\n)=0$.
The latter identity can be equivalently written as $\n=\r_\kappa(\n)={\kappa}/{\left\{1+\kappa-\p(\n)\right\}}$.
As a consequence, we infer
\begin{equation}
	\kappa(\n)=\frac{\n\left\{\p(\n)-1\right\}}{\n-1}
		\qquad\textrm{for}\quad \n\neq 1\,.
\end{equation}
and thus, since \eqref{eq:hyppressure} holds, 
\begin{equation*}
	\lim_{\n\to 0^+} \kappa(\n)=0,\qquad 
	\lim_{\n\to 1}  \kappa(\n)=-1,\qquad
	\lim_{\n\to +\infty}  \kappa(\n)=+\infty.
\end{equation*}
In addition,  $\n\mapsto \kappa(\n)$ is differentiable with respect to $\n$ for $\n\neq 1$ with derivative
\begin{equation*}
	\kappa'(\n)
	=\frac{\n(\n-1)\p'(\n)+1-\p(\n)}{(\n-1)^2}\,.
\end{equation*}
Then, applying de l'H\^opital rule, we infer
\begin{equation*}
	\lim_{\n\to 1}  \kappa'(\n)
		=\lim_{\n\to 1}\frac{2\p'(\n)+\n\p''(\n)}{2}=\kappa_\ast+\frac12\,\p''(1)>0\,,
\end{equation*}
showing that $\kappa\in C^1(0,+\infty)$.
Moreover, the numerator in the expression for the derivative  $\kappa'(n)$ is positive, 
since it vanishes at $\n=1$ and a further differentiation gives
\begin{equation*}
	\frac{\ud}{\ud\n}\left\{\n(\n-1)\p'(\n)+1-\p(\n)\right\}=(\n-1)\left\{2\p'(\n)+\n\p''(\n)\right\}
\end{equation*}
which is of the same sign as $\n-1$ and so $\n\mapsto \kappa(\n)$ is increasing.

Finally, 
 thanks to the strict positivity of $\p'$,
 $\p(\n)-1$ is of the same sign as $\n-1$ and, since  $\kappa(\n)$ can be rewritten as
\begin{equation*}
	\kappa(\n)=\p(\n)-1+\frac{\p(\n)-1}{\n-1},
\end{equation*}
we deduce that $\kappa(\n)>\p(\n)-1$.
Therefore, $\n<\bar{\n}(\kappa(\n))$ and we conclude that 
there exists a unique $\kappa$ such that $\g(\n)=0$. 
\end{proof}

\begin{figure}[thb] 
   \tikzset{mark size=3.5}   
  \begin{tikzpicture}[thick,scale=0.5, every node/.style={scale=1.3}] 
    \input{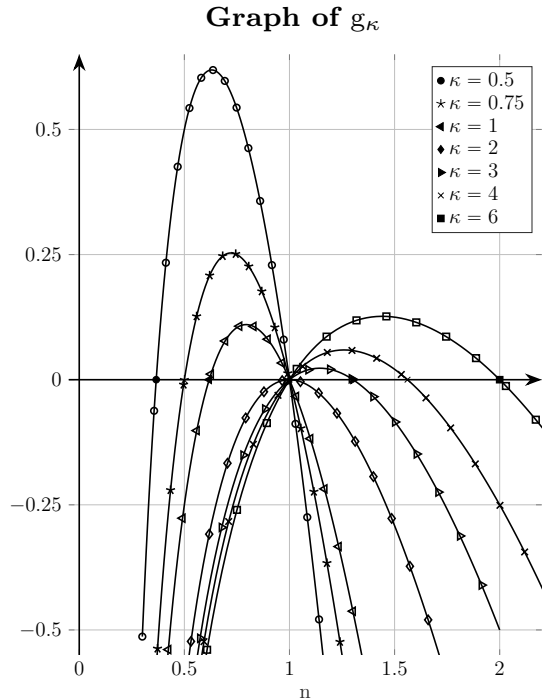}         
  \end{tikzpicture} 
\caption{\footnotesize The graph of the function $\g$ in \eqref{eq:rescaledODE} for the $\gamma$-law \eqref{eq:gammalaw}
with exponent $\gamma=2$.
The markers are the same as in Figures \ref{fig:ngraph} and \ref{fig:orbits}.}
\label{fig:g_plot}
\end{figure}

Let the function $\h$ be defined by
\begin{equation} 
	\h(\n):=\left\{\n \p(\n)\right\}'=\p(\n)+\n\p'(\n)\,.
\end{equation}
In particular, because $\h'=2\p'+\p''>0$, the function $\h$ is strictly increasing together with its inverse $\h^{-1}$.
Then, the following function is well-defined for any $\kappa>0$
\begin{equation}\label{eq:def:nhashtag}
	\n_\#(\kappa):=\h^{-1}(1+\kappa)\in(1,\infty)\,.
\end{equation}
Being $\h(1)=\p(1)+\p'(1)=1+\kappa_\ast$, there holds $\n_\#(\kappa_\ast)=1$.

\begin{lemma}\label{lem:tauhashtag}
Given $\kappa>0$, let $\n_\#=\n_\#(\kappa)$ be defined as in \eqref{eq:def:nhashtag}.
The function 
\begin{equation}\label{eq:def_tauhashtag}
	\tau_\#(\kappa):=\frac{\kappa}{\n_\#^2 \p'(\n_\#)}
\end{equation}
is such that $0<\tau_\#(\kappa)\leq 1$ for all $\kappa>0$ and $\tau_\#(\kappa)=1$ if and only if $\kappa=\kappa_\ast$.\par
Moreover, $\tau_\#=\tau_\#(\kappa)$ tends to $0$ as $\kappa\to 0^+$ and as $\kappa\to +\infty$.
\end{lemma}

\begin{proof}
To begin, let us observe that $\n_\#(\kappa_\ast)=1$ and $\p'(1)=\kappa_\ast$ so that $\tau_\#(\kappa_\ast)=1$.
The positivity of $\tau_\#$ being obvious, let us show that $\tau_\#\leq 1$ for $\kappa>0$,
with the equality holding only if $\kappa=\kappa_\ast$.
Indeed, the above inequality is equivalent to
\begin{equation}\label{eq:tauhashtag<1}
	\f(\kappa):=\n_\#^2 \p'(\n_\#)-\kappa\geq 0\qquad \forall\,\kappa>0\,.
\end{equation}
Note that $\f(\kappa_\ast)=\p'(1)-\kappa_\ast=0$.
Differentiating with respect to $\kappa$, we infer  
\begin{equation*}
	\f'(\kappa)=\left\{2 \p'(\n_\#)+\n_\# \p''(\n_\#)\right\}\n_\# \n_\#'-1
	=\frac{\left\{2 \p'(\n_\#)+\n_\# \p''(\n_\#)\right\}\n_\#}{\h'(\h^{-1}(1+\kappa))}-1
	=\n_\#-1\,.
\end{equation*}
Differentiating again, since $\n_\#'=1/\h'(\n_\#)>0$, we conclude
that $\f$ is strictly convex, its unique minimum being $0$ at $\kappa=\kappa_\ast$.              
As a consequence, inequality \eqref{eq:tauhashtag<1} holds.

Next, let us observe that $\n_\#(0)=\h^{-1}(1)>\h^{-1}(0)=0$ since $\h(0)=\p(0)=0$ and $\h^{-1}$ is strictly increasing.
Hence,  $\frac{\tau_\#(\kappa)}{\kappa}$ tends to a strictly positive number as $\kappa\to 0^+$ and the limit
of $\tau_\#$ at $\kappa=0$ is identified.

Concerning the behavior at $+\infty$, since $\h(+\infty)=+\infty$, there holds $n_\#(+\infty)=+\infty$,
Then, applying de l'H\^opital rule, we obtain
\begin{equation*}
	\lim_{\kappa\to +\infty} \tau_\#(\kappa)=\lim_{\kappa\to +\infty} \frac{1}{\{\n_\#^2 \p'(\n_\#)\}'}
		=\lim_{\kappa\to +\infty}\frac{1}{n_\#}\ \frac{\h'}{2\p'+\n_\#\p''}
		=\lim_{\kappa\to +\infty} \frac{1}{\n_\#}=0,
\end{equation*}
completing the proof.
\end{proof}

\begin{theorem}\label{th:singlimit_existence}
Given $\kappa>0$ with $\kappa\neq \kappa_\ast$, let
$n_\times=n_\ast \n_\times(\kappa)\neq n_\ast$  be the equilibrium value
  defined thanks to
$\n_\times(\kappa)$ the solution given by Lemma \ref{lem:ncross}.
Then, if $\tau<\tau_\#(\kappa)$, problem \eqref{eq:odeth0}-\eqref{eq:constraint} admits monotone solutions $\y\mapsto n(\y)$ connecting
asymptotically $n_\times$ to $n_\ast$ with monotonicity related to the sign of $u_\ast$.
\end{theorem}

\begin{rmk}The 
definition of the parameters has practical consequences, for instance for numerical purposes.
Choosing $u_\ast$, $\tau$ and $\kappa$ leads to inverting $n\mapsto \frac{n}{p(n)}$ in order to retrieve $n_\ast$, which might require additional 
assumptions on the pressure law, hopefully satisfied by the $\gamma$-law.
\end{rmk}

\begin{figure}[bht]
  \begin{tikzpicture}[thick,scale=0.6, every node/.style={scale=1.3}] 
%
%

\begin{axis}[
width=6in,
height=3in,
at={(0.822in,0.846in)},
xmin=0,
xmax=12,
ymin=0, 
ymax=1.1,
scale only axis,
area style, 
hide axis
]
    \addplot[name path=f,domain=0:11.5,black] table[row sep=crcr]{%
0	0\\
0.121212121212121	0.265256638446294\\
0.242424242424242	0.454802725126372\\
0.363636363636364	0.593295878967653\\
0.484848484848485	0.696201122173728\\
0.606060606060606	0.773615295185808\\
0.727272727272727	0.832353508623751\\
0.848484848484849	0.877146500790414\\
0.96969696969697	0.911357558260551\\
1.09090909090909	0.937427567679419\\
1.21212121212121	0.957160036312575\\
1.33333333333333	0.97190864488087\\
1.45454545454545	0.982703641586785\\
1.57575757575758	0.990338838818667\\
1.6969696969697	0.995432631608339\\
1.81818181818182	0.998471524977282\\
1.93939393939394	0.99984166131695\\
2.06060606060606	0.9998519740466\\
2.18181818181818	0.998751407096802\\
2.3030303030303	0.996741869382683\\
2.42424242424242	0.993988084105178\\
2.54545454545455	0.990625150316092\\
2.66666666666667	0.986764400436317\\
2.78787878787879	0.982497975570433\\
2.90909090909091	0.977902426920707\\
3.03030303030303	0.973041570977147\\
3.15151515151515	0.967968768255082\\
3.27272727272727	0.962728753323514\\
3.39393939393939	0.957359113060527\\
3.51515151515152	0.951891487280619\\
3.63636363636364	0.946352548870404\\
3.75757575757576	0.940764807772366\\
3.87878787878788	0.935147273453992\\
4	0.92951600308978\\
4.12121212121212	0.923884556985664\\
4.24242424242424	0.918264378365328\\
4.36363636363636	0.91266511120341\\
4.48484848484848	0.907094867100344\\
4.60606060606061	0.901560450074476\\
4.72727272727273	0.896067546469019\\
4.84848484848485	0.890620885835964\\
4.96969696969697	0.885224377591058\\
5.09090909090909	0.879881227376017\\
5.21212121212121	0.874594036371858\\
5.33333333333333	0.869364886246339\\
5.45454545454545	0.864195411962139\\
5.57575757575758	0.859086864299812\\
5.6969696969697	0.854040163644103\\
5.81818181818182	0.849055946330982\\
5.93939393939394	0.844134604645427\\
6.06060606060606	0.839276321388305\\
6.18181818181818	0.83448109978812\\
6.3030303030303	0.829748789414591\\
6.42424242424242	0.825079108651809\\
6.54545454545455	0.820471664205539\\
6.66666666666667	0.815925968049423\\
6.78787878787879	0.811441452156\\
6.90909090909091	0.807017481308783\\
7.03030303030303	0.802653364249631\\
7.15151515151515	0.798348363379976\\
7.27272727272727	0.794101703204169\\
7.39393939393939	0.789912577677376\\
7.51515151515152	0.785780156598426\\
7.63636363636364	0.781703591169129\\
7.75757575757576	0.777682018825447\\
7.87878787878788	0.773714567432007\\
8	0.769800358919501\\
8.12121212121212	0.765938512434249\\
8.24242424242424	0.762128147060308\\
8.36363636363636	0.758368384166832\\
8.48484848484848	0.75465834942676\\
8.60606060606061	0.750997174547114\\
8.72727272727273	0.747383998746219\\
8.84848484848485	0.743817970008769\\
8.96969696969697	0.740298246145907\\
9.09090909090909	0.736823995684149\\
9.21212121212121	0.73339439860412\\
9.33333333333333	0.730008646947537\\
9.45454545454546	0.726665945308668\\
9.57575757575758	0.723365511224573\\
9.6969696969697	0.720106575476705\\
9.81818181818182	0.71688838231501\\
9.93939393939394	0.713710189614313\\
10.0606060606061	0.71057126897164\\
10.1818181818182	0.707470905752134\\
10.3030303030303	0.704408399090286\\
10.4242424242424	0.701383061852468\\
10.5454545454545	0.698394220566006\\
10.6666666666667	0.695441215319465\\
10.7878787878788	0.69252339963825\\
10.9090909090909	0.689640140339141\\
11.030303030303	0.686790817366978\\
11.1515151515152	0.683974823616332\\
11.2727272727273	0.681191564740639\\
11.3939393939394	0.678440458951009\\
11.5151515151515	0.675720936806662\\
};

    \path[name path=axis1] (axis cs:0,0) -- (axis cs:2,0);
    \path[name path=axis2] (axis cs:2,0) -- (axis cs:11.5,0);

    \addplot [
        thick,
	color=gray,
        fill=gray, 
        fill opacity=0.05
    ]
    fill between[
        of=f and axis1,
        soft clip={domain=0:2},
    ];
    \addplot [
        thick,
	color=gray,
        fill=gray, 
        fill opacity=0.15
    ]
    fill between[
        of=f and axis2,
        soft clip={domain=2:11.5},
    ];
\end{axis}

\begin{axis}[%
width=6in,
height=3in,
at={(0.822in,0.846in)},
scale only axis,
area style,
stack plots=y,
clip=false,
xmin=0,
xmax=12,
xtick={\empty},
xlabel style={font=\color{white!15!black}},
xlabel={$\kappa$},
ymin=0,
ymax=1.1,
ytick={0, 1},
ylabel style={font=\color{white!15!black}},
ylabel={$\tau$},
axis lines = left,
axis line style={-{Stealth[scale=2]}},
]
\addplot [color=black, dotted, line width=1.5pt]
  table[row sep=crcr]{%
2	0\\
2	1\\
};
\node[right, align=left]
at (axis cs:1.8,-0.04) {$\kappa^*$};
\node[right, align=left]
at (axis cs:0.09,0.1) {$n_\times<n^*$};
\node[right, align=left]
at (axis cs:5,0.5) {$n_\times>n^*$};
\end{axis}

 \node at (8, 11)  (c)     {\bf\tiny Graph of $\tau_\#$};

      
    \end{tikzpicture}%
\caption{\footnotesize The graph of the function $\tau_\#$ in the case of the $\gamma$-law \eqref{eq:gammalaw} with $\gamma=2$.
Small shocks are concentrated in a neighborhood of $\kappa=\kappa_\ast=2$.}
\label{fig:tau_plot}             
\end{figure}
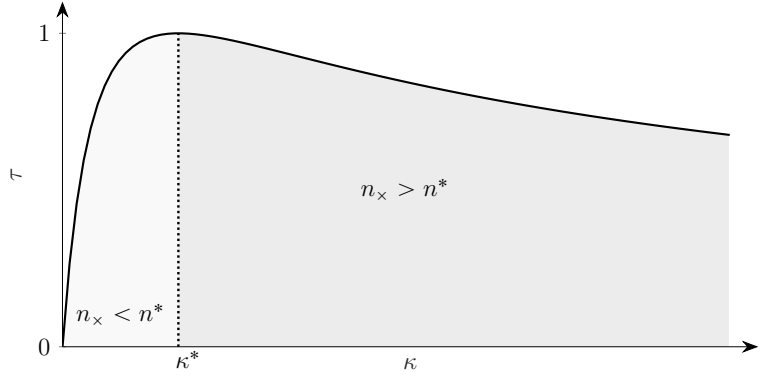  
  
\begin{proof}
For $\r-\tau\n\neq 0$ and introducing the new variable $\z$ such that
\begin{equation}\label{eq:def_abpsi}
	\frac{\ud}{\ud\z}=\frac{u_\ast\left(\r-\tau\n\right) \p'(\n)}{\kappa\r^2}\frac{\ud}{\ud \y}\,,
\end{equation}
problem \eqref{eq:rescaled} becomes
\begin{equation}\label{eq:rescaledODE}
	\frac{\ud \n}{\ud \z} = \g(\n).
\end{equation}
where $\g$ is defined as in Lemma  \ref{lem:ncross}.
A straightforward argument, based on the analysis of the sign of function $\g$, shows the existence
of the heteroclinic connection between $1$ and $\n_\times$ for \eqref{eq:rescaledODE} for $\kappa\neq \kappa_\ast$,
whenever $\r-\tau\n>0$.

The threshold level $\tau_\#$ appears as a consequence of the constraint $\r>\tau \n$,
indicating that the curve $(\n,\r_\kappa)$ lies above $(\n,\tau \n)$. 
Differentiating $\r_\kappa$ with respect to $\n$, we infer
\begin{equation*}
	\frac{\ud\r_\kappa}{\ud\n}=\frac{\kappa \p'(\n)}{\left\{1+\kappa-\p(\n)\right\}^2}\,,
\end{equation*}
which is positive and increasing for the properites of $\p$. 
In particular, $\r_\kappa$ is convex in $\left(0,\bar{\n}(\kappa)\right)$ where  $\bar{\n}(\kappa)$ has been introduced right after \eqref{eq:rescaled}.

Next let us look for the pair $(\n_\#,\tau_\#)$ such that the tangent to the graph of the functions $\r_\kappa$
is given by the straight line $\r=\tau_\# \n$.
This amounts to searching the solutions of 
\begin{equation*}
	\r_\kappa(\n_\#)=\frac{\kappa}{1+\kappa-\p(\n_\#)}=\tau_\# \n_\#\
	\quad\textrm{and}\quad
	\r_\kappa'(\n_\#)=\frac{\kappa \p'(\n_\#)}{\left\{1+\kappa-\p(\n_\#)\right\}^2}=\tau_\#\,.
\end{equation*}
Replacing the first identity into the second and simplifying, we get 
 ${\p(\n_\#)+\n_\# \p'(\n_\#)=1+\kappa}$. 
Then, we immediately recognize that $\n_\#=\h^{-1}(1+\kappa)$ and $\tau_\#=\r_\kappa(\n_\#)/\n_\#=\r_\kappa'(\n_\#)$
which corresponds to the value defined in \eqref{eq:def_tauhashtag}.
Note also that $\g(n_\#)=(1-\tau_\#)/\n_\#\geq 0$.
Summarizing, for $\tau\in(0,\tau_\#)$ the constraint $\r_\kappa>\tau\n$ is always
satisfied and the change of variables is legit.

Conversely, for $\tau\in(\tau_\#,1)$ there exist two values
$\n_\ell, \n_r\in (0,\bar{\n})$ with $\n_\ell<\n_r$ and $\r_\kappa(\n_{\ell,r})=\tau \n_{\ell,r}$, such that
the condition $\r_\kappa>\tau\n$ holds if and only if $\n\in(0,\n_\ell)$ or $\n\in(\n_r,\bar{\n})$.
In addition, for $\tau>\tau_\#$, we have
\begin{equation*}
	\g(\n_{\ell,r})=\frac{1}{\r_\kappa(\n_{\ell,r})}-\frac{1}{\n_{\ell,r}}=\frac{1+\kappa-\p(\n_{\ell,r})}{\kappa}-\frac{1}{\n_{\ell,r}}
		=\frac{1-\tau}{\tau \n_{\ell,r}}>0\,,
\end{equation*}
so that for $\kappa<\kappa_\ast$, there holds $0<\n_\times<\n_{\ell}<\n_{r}<1<\bar{\n}$.
In particular, for $\tau>\tau_\#$ and $\kappa<\kappa_\ast$,  the function $\varphi(\n):=\r_\kappa(\n)-\tau \n$
is negative in the interval $(\n_\ell,\n_r)\subset(\n_\times,1)$.
Similarly, for $\kappa>\kappa_\ast$, $\varphi$ is negative in ${(\n_\ell,\n_r)\subset(1,\n_\times)}$.
In both cases, the change of variables \eqref{eq:def_abpsi} is not applicable and existence of the connection is precluded
since the phyisical requirement $\rho>0$ is violated.
\end{proof}

\begin{rmk}\rm
Figure \ref{fig:ngraph} shows the profiles $\n$ (respectively, $\r$) connecting $1$ to $\n_\times$
(resp. $1$ to $\r_\kappa(\n_\times)$) associated to several values of $\kappa$ for the choice $\p(\n)=\n^2$,
illustrating the increasing character of the equilibrium map $\n_\times=\n_\times(\kappa)$.
This point is emphasized in Figure \ref{fig:orbits} in the phase portrait corresponding to the same values of $\kappa$,
showing that the orbits are convex.
Also, note that $\n_\times$ and $\r_\kappa(\n_\times)$ do not depend on $\tau$, but the profiles $n$ and $r$ do,
through $r_*=n_*/\tau$.
The condition $r-n=r_*(\r-\tau \n)>0$, with $\tau<\tau_\#$, shows as $\n\mapsto \tau_\#\n$ is tangent to the orbit
at the point $(\n_\#,\r_\kappa(\n_\#))$.
\end{rmk}

 \begin{figure}
   \tikzset{mark size=3.5}  
 \begin{tikzpicture}[thick,scale=0.5, every node/.style={scale=1.0}]
    \input{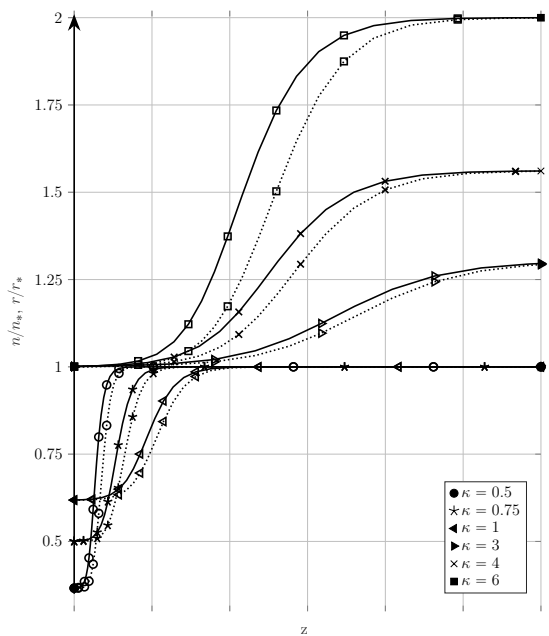}         
  \end{tikzpicture} 
\caption{\footnotesize Graphs of $n/n_\ast$ (plain) and
  $r/r_\ast=\r_\kappa(n/n_\ast)$ (dotted) where $n$ solves problem \eqref{eq:odeth0}-\eqref{eq:constraint}
 for several values of $\kappa$ such that $\tau<\tau_\#$.
The pressure law $p$ is the $\gamma$-law \eqref{eq:gammalaw} with exponent $\gamma=2$.} 
\label{fig:ngraph}
\end{figure}
                                  
\begin{figure}
  \tikzset{mark size=3.5}  
  \begin{tikzpicture}[thick,scale=0.5, every node/.style={scale=1.3}] 
    \input{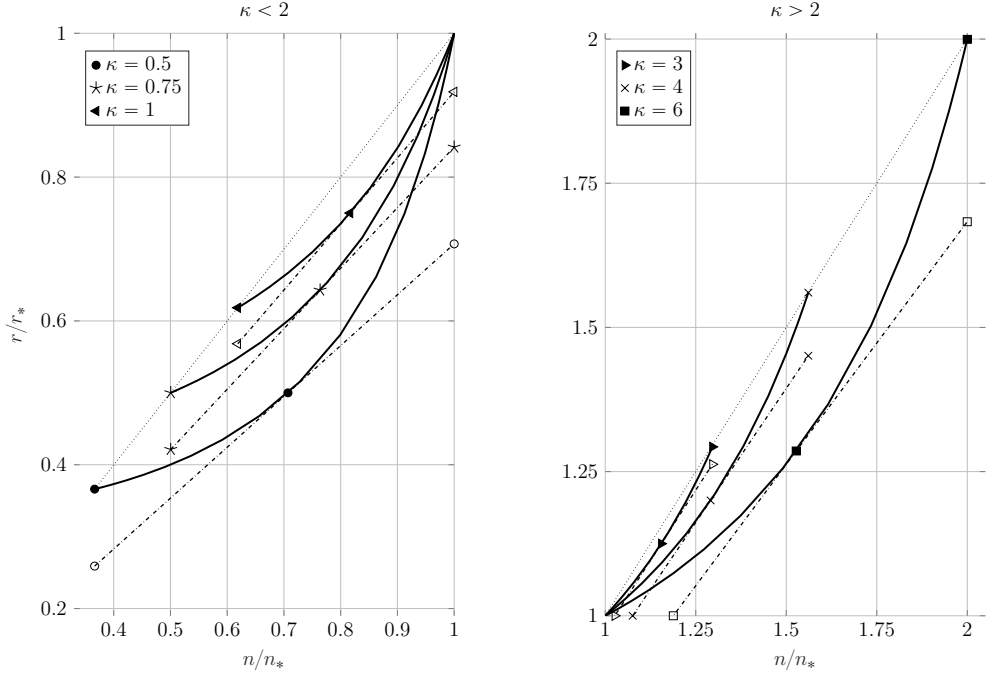}         
  \end{tikzpicture}  
\caption{\footnotesize Orbits connecting $\left(\n_\times,\r_\kappa(\n_\times)\right)$ to $(1,1)$
for several values of $\kappa$ for the $\gamma$-law \eqref{eq:gammalaw} with exponent $\gamma=2$.
The straight line $n\mapsto \tau_\#(\kappa) n$ is plotted for each value of $\kappa$, and the tangent point
with the corresponding orbit is indicated.
The markers are the same as in Figure \ref{fig:ngraph}.}
\label{fig:orbits}
\end{figure}  
 
Let us also observe that, since $\p(1)=1$, there holds $\r_\kappa(1)-\tau=1-\tau>0$.
Hence, small shocks are always admissible also in the case of zero-temperature.

\begin{exo}\rm     
For the sake of concreteness, let us again consider the pressure given by the $\gamma$-law \eqref{eq:gammalaw}.
Incidentally, let us note that $\kappa_\ast=\gamma$ for any positive constant $C$.
Then, most auxiliary functions can be determined giving the explicit expressions
\begin{equation*}
	\p(\n)=\n^\gamma,\qquad \h(\n)=(1+\gamma)\n^\gamma,\qquad
        \h^{-1}(\r)=\left(\frac{\r}{1+\gamma}\right)^{1/\gamma}.  
\end{equation*}           
Moreover, there holds
\begin{equation*}
	\n_\#(\kappa)=\left(\frac{1+\kappa}{1+\gamma}\right)^{1/\gamma}
	\quad\textrm{and}\quad
	\tau_\#(\kappa)=\frac{(1+\gamma)^{1+1/\gamma}}{\gamma}\ \frac{\kappa}{(1+\kappa)^{1+1/\gamma}}
\end{equation*}
In the special case $\gamma=2$, the function $\g$ is a rational function whose factorization is
\begin{equation*}          
	\begin{aligned} 
	\g(\n)&=\frac{1+\kappa-\n^2}{\kappa}-\frac{1}{\n}=-\frac{\n^3-(1+\kappa)\n+\kappa}{\kappa \n}=-\frac{1}{\n}(\n+\n_-)(\n-1)(\n-\n_\times)\,,
	\end{aligned}
\end{equation*} 
where $\n_-$ and $\n_\times$ are given by
\begin{equation*}
	\n_-:=\tfrac{1}{2}\left\{(1+4\kappa)^{1/2}+ 1\right\},\qquad
	\n_\times:=\tfrac{1}{2}\left\{(1+4\kappa)^{1/2}- 1\right\}.
\end{equation*}
Corresponding graphical representations of the function $\varphi$ (defined at the very end of proof of Theorem \ref{th:singlimit_existence}) 
for different choices of $\tau$ are given in Figure~\ref{fig:g_plot_bis}.
Here, the limiting value $\tau_{\#}$ is equal to $1$ at $\kappa=\gamma=2$ and is explicitly represented to show tangency of the graph
with the horizontal axis. 
Above this $\kappa$-dependent threshold value, the still existing heteroclinic connection from $\n_\times$ and $1$
(corresponding to the connection from $n_\times$ and $n_\ast$) is not physically admissible since the carrier phase $\rho$
is negative in a neighborhood of both asymptotic states. 
\begin{figure}
  \tikzset{mark size=3.5}  
  \begin{tikzpicture}[thick,scale=0.5, every node/.style={scale=1.3}] 
    \input{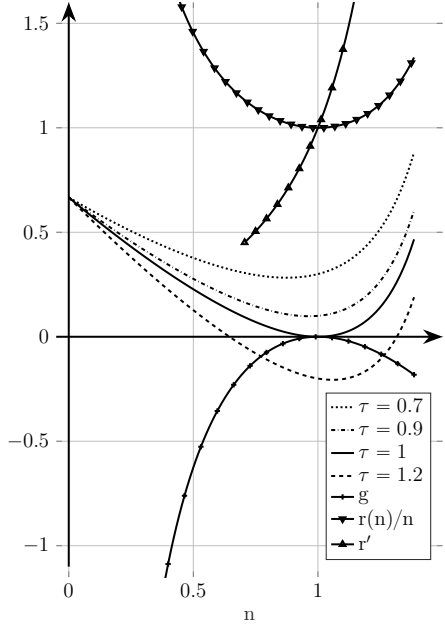}          
  \end{tikzpicture} 
\caption{\footnotesize Graphs of functions $\varphi(\n)=\r_\kappa(\n)-\tau \n$ in the case of $\gamma$-law \eqref{eq:gammalaw} 
with $\gamma=\kappa=2$ and various choices of $\tau$.
The graph of $\g$ (with opposite convexity) has been superposed for comparison,
as well as the maps $\n\mapsto \r_\kappa(\n)/\n$ and $\n\mapsto\r'(\n)$ that intersect at $(\n_\#,\tau_\#)$.}
\label{fig:g_plot_bis}
\end{figure}     
\end{exo} 

\subsection{Further scrutiny for positive temperature}
        
Next, we focus on the case $\theta>0$ with the intention of grasping information from the singular limit behavior $\theta=0$.
System \eqref{eq:ode_profile} can be equivalently written as
\begin{equation}\label{eq:postemp} 
	\left\{\begin{aligned}
	&\frac{r^2}{\rho\, p'+\theta n} \frac{\ud n}{\ud\y} =\frac{w_\ast}{n}\left( \frac{n}{r}-\frac{n_\ast}{r_\ast} \right)
		 -\frac{w_\ast n}{\rho}\left\{\frac{1}{r}-\frac{1}{r_\ast}+\frac{p-p_\ast+\theta(\rho-\rho_\ast)}{w_\ast^2}\right\}\,,\\
	&\theta\frac{\ud r}{\ud\y} =-\frac{w_\ast r^2}{\rho}\left\{\frac{1}{r}-\frac{1}{r_\ast}+\frac{p-p_\ast+\theta(\rho-\rho_\ast)}{w_\ast^2}\right\}\,.
	\end{aligned}\right.
\end{equation}
where $w_\ast=r_\ast u_\ast$.
Next, with same notation as before for $\n$, $\r$, $\tau$, $\kappa$, $\p$ (see \eqref{eq:variousdef:1}-\eqref{eq:variousdef}) and observing that $\rho= r_\ast(\r-\tau \n)$, 
we set $\T$ as in \eqref{eq:rescaled} and $\mathscr{B}_\eps:=\mathscr{B}^0+\eps \mathscr{B}^1$ where 
\begin{equation*}
	\mathscr{B}^0(\n,\r):=\frac{1+\kappa-\p(\n)}{\kappa}-\frac{1}{\r}=\g(\n)-\T(\n,\r),\qquad 
        \mathscr{B}^1(\n,\r):=\frac{1-\tau-(\r-\tau\n)}{\kappa}\,.
\end{equation*}
and $\eps:={\kappa\theta}/{u_\ast^2}={r_\ast\theta}/{p_\ast}$.
Then, the renormalized version of system \eqref{eq:postemp} is
\begin{equation*}
	\left\{\begin{aligned}
	u_{\ast} \frac{\ud \n}{\ud\y} &= \frac{\kappa \r^2}{(\r-\tau\n) \p'(\n)+\varepsilon\tau^2\n}\left\{\T(\n,\r)+\frac{\tau \n}{\r-\tau \n}\mathscr{B}_\eps(\n,\r)\right\}\,,\\
	\eps u_\ast\frac{\ud \r}{\ud\y} &= \frac{\kappa \r^2}{\r-\tau
          \n}\mathscr{B}_\eps(\n,\r)\,, 
	\end{aligned}\right.
\end{equation*}
Note that $\mathscr{B}_\eps$, varying linearly with respect to $\varepsilon$, depends also on $\p$ (through $\mathscr{B}_0$),
on $\tau$ (through $\mathscr{B}_1$) and on $\kappa$ (through both $\mathscr{B}^0$ and $\mathscr{B}^1$).
In addition, we remark that the parameter $\eps$ is small also in cases where $\theta$ is of order 1
and  $\kappa/u_\ast^2$ is small.

Introducing the new variable $\z$ such that
\begin{equation}  
	\frac{\ud}{\ud\z}=\frac{u_\ast(\r-\tau\n) \p'(\n)+\eps \tau^2\n}{\kappa\r^2}\frac{\ud}{\ud \y}\,,
\end{equation}
we arrive at the final expression
\begin{equation}\label{eq:rescaled_pos_temp}
	\left\{\begin{aligned}
	\frac{\ud \n}{\ud \z} &= \T(\n,\r)+\frac{\tau \n}{\r-\tau \n}\,\mathscr{B}_\eps(\n,\r)\,,\\
	\eps \frac{\ud \r}{\ud \z} &= \left(\p'(\n)+\frac{\eps\tau^2\n}{\r-\tau\n}\right) \mathscr{B}_\eps(\n,\r)\,.
	\end{aligned}\right.
\end{equation}
Preliminarily, observe that, since $\p(1)=1$ and thus $\mathscr B_\eps(1,1)=0$, the pair $(\n,\r)=(1,1)$
defines an equilibrium solution  for \eqref{eq:rescaled_pos_temp} for any $\varepsilon$, $\tau$ and $\kappa$.

\begin{prop}
 Let $\kappa>0$ and $\tau\in(0,1)$. Then, for any $\eps>0$ there exists a unique equilibrium point
 $(n_\times^\eps,r_\times^\eps)\neq(n_\ast,r_\ast)$ of system
 \eqref{eq:postemp}. Denoting $\n_\times^\eps=n_\times^\eps/n_\ast$ and
 ${\r_\times^\eps=r_\times^\eps/r_\ast}$, we have
 $\mathscr{T}(\n_\times^\eps,\r_\times^\eps)=\mathscr B_\eps(\n_\times^\eps,\r_\times^\eps)=0$.
 Moreover, the two coefficients $\n_{\times,0}$ and $\n_{\times,1}$ of the first
 order Taylor expansion of $\n_\times^\eps$ at $\eps=0$, 
viz. $\n_\times^\eps=\n_{\times,0}+\eps \n_{\times,1}+\mathscr{O}(\eps^2)$, are
\begin{equation}\label{eq:exp_ncross_explicit}
	\n_{\times,0}=\n_\times\qquad\textrm{and}\qquad
	\n_{\times,1}=\dfrac{1-\tau}{\kappa}\ \dfrac{\n_\times-1}{\g'(\n_\times)}<0\,,
\end{equation}
where $n_\times$ is the equilibrium of system \eqref{eq:rescaled} (as described in Lemma \ref{lem:ncross}).
\end{prop}

\begin{proof}
The pair $(\n_\times^\eps,\r_\times^\eps)$ solves $\T(\n_\times^\eps,\r_\times^\eps)=\mathscr{B}_\eps(\n_\times^\eps,\r_\times^\eps)=0$
which is equivalent to
\begin{equation}\label{eq:equilibria}
  \left\{\begin{aligned}
      \r_\times^\eps&=\n_\times^\eps,\\
      \g(\n_\times^\eps)&=\frac{\eps(1-\tau)}{\kappa}(\n_\times^\eps-1)\,,
    \end{aligned}\right.
\end{equation}
referring to the notation in Lemma \ref{lem:ncross}. 
Since $\g(\n_\times)=0$, the zero-th order $\n_{\times,0}$ in the expansion with respect to $\eps$
of the solution $\n_\times^\eps$ coincides with $\n_\times$.
Moreover, the first order coefficient $\n_{\times,1}$ can be obtained from \eqref{eq:equilibria}
by substitution of the expansion and cancellation of the common coefficient $\eps$, that is
\begin{equation}\label{eq:exp_ncross}
	\n_{\times,1}\g'(\n_\times)=\dfrac{1-\tau}{\kappa}(\n_\times-1)\,.
\end{equation}
which gives the desired equality.
Note that $\g'(\n_\times)$ cannot vanish simultaneously with $\g(\n_\times)$
 since $\g'$ is strictly decreasing --see \eqref{eq:gprime}-- and $\g(1)=\g(\n_\times)=0$.

Finally, $\g'$ being decreasing yields
\begin{equation*}
	\frac{\g'(\n_\times)}{\n_\times-1}=\frac{\g'(\n_\times)-\g'(1)}{\n_\times-1}<0,
\end{equation*}
and thus $\n_{\times,1}$ is negative. 
\end{proof}

\begin{figure}
   \tikzset{mark size=3.5}  
  \begin{tikzpicture}[thick,scale=0.5, every node/.style={scale=1.3}] 
    \input{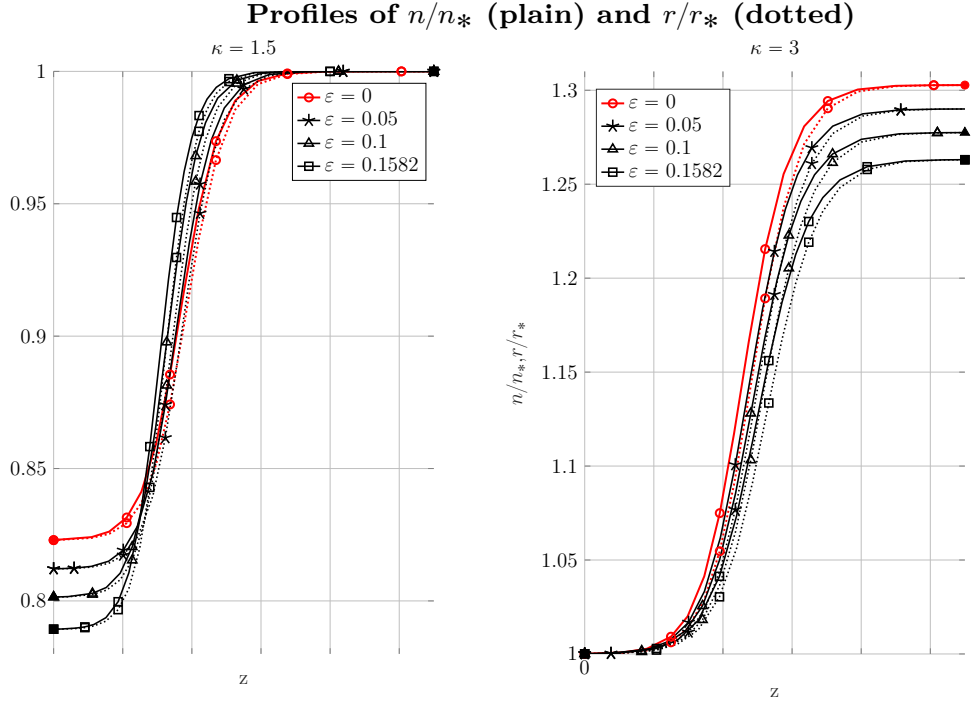}         
  \end{tikzpicture} 
\caption{\footnotesize Graphs of $\n$ (plain) and $\r(n)$ (dotted) where $\n$ solves  problem \eqref{eq:rescaled_pos_temp}
for several values of $\eps$, $n_\ast$ being fixed and $\tau=n_\ast/r_\ast$ being chosen strictly less 
than $\tau_\#(\kappa)$ (here, $\tau=0.3\,\tau_\#$). The pressure law is the $\gamma$-law \eqref{eq:gammalaw}
with exponent $\gamma=2$.} 
\label{fig:prof_nr} 
\end{figure}

\begin{figure}
  \tikzset{mark size=3.5}  
  \begin{tikzpicture}[thick,scale=0.5, every node/.style={scale=1.3}] 
    \input{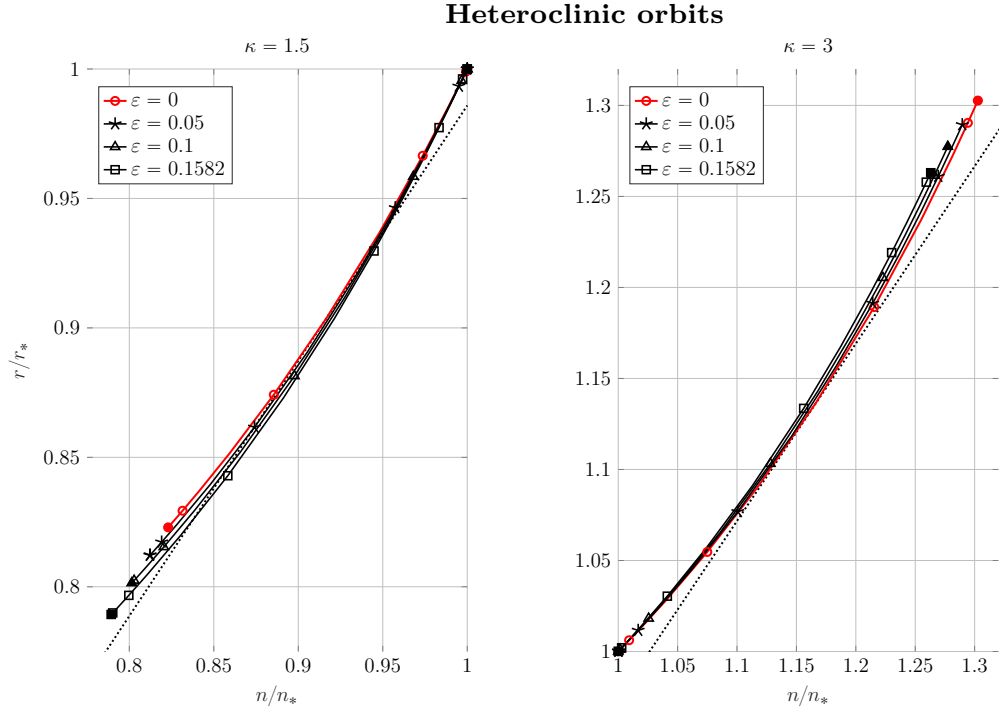}         
  \end{tikzpicture} 
\caption{\footnotesize Orbits connecting $(\n_\times,\r(\n_\times))$ to $(1,1)$ for several values of $\eps$
for the $\gamma$-law \eqref{eq:gammalaw} with exponent $\gamma=2$.
The straight line $n\mapsto \tau_\#(\kappa) n$ is plotted for reference, and the tangent point with the corresponding
orbit is indicated (markers as in Figure \ref{fig:prof_nr}).}
\label{fig:orbits_nr} 
\end{figure} 

\begin{rmk}\rm
The formation of viscous profiles joining monotonically the equilibrium values is shown in Figure~\ref{fig:prof_nr},
while Figure~\ref{fig:orbits_nr} represents the corresponding heteroclinic orbits in the $(n,r)$ plane.
The fact that $\n_\times^\eps<\n_\times$ for small values of $\eps$ is showing well. 
 
These numerical results are given with a purpose reduced to an illustration of the previous discussion,
showing a computational evidence for the existence of viscous shock profiles. 
However, the apparent convexity of the orbits is worth investigating, as is the fact that the sign of $\n_1$
seems to imply that, if $\kappa=3$, $\tau$ might be chosen closer to $\tau_\#$.
\end{rmk}

Capturing viscous profiles is very sensitive because it requires  the determination of the 
equilibrium value with high accuracy.
Again, the resolution of the differential system should be performed with a high-order method
in order to capture the profile.
A thorough numerical investigation will be presented  elsewhere, addressing in further details
the computational difficulties and the role of the parameters of the model.

\bibliographystyle{siam}
\bibliography{./FPS}

\end{document}